\documentclass[12pt,oneside,english]{amsart}
\usepackage[T1]{fontenc}
\usepackage[latin9]{inputenc}
\usepackage{geometry}
\usepackage{amstext}
\usepackage{amsthm}
\usepackage{amssymb}
\usepackage{stmaryrd}
\usepackage[all]{xy}

\makeatletter
\numberwithin{equation}{section}
\numberwithin{figure}{section}
\theoremstyle{plain}
\newtheorem{thm}{\protect\theoremname}[section]
\theoremstyle{definition}
\newtheorem{defn}[thm]{\protect\definitionname}
\theoremstyle{plain}
\newtheorem{lem}[thm]{\protect\lemmaname}
\theoremstyle{plain}
\newtheorem{cor}[thm]{\protect\corollaryname}
\theoremstyle{remark}
\newtheorem{rem}[thm]{\protect\remarkname}
\theoremstyle{plain}
\newtheorem{prop}[thm]{\protect\propositionname}

\usepackage{hyperref}

\makeatother

\usepackage{babel}
\providecommand{\corollaryname}{Corollary}
\providecommand{\definitionname}{Definition}
\providecommand{\lemmaname}{Lemma}
\providecommand{\propositionname}{Proposition}
\providecommand{\remarkname}{Remark}
\providecommand{\theoremname}{Theorem}

\begin{document}
\title[Motivic versions of mass formulas]{Motivic versions of mass formulas by Krasner, Serre and Bhargava}
\author{Takehiko Yasuda}
\address{Department of Mathematics, Graduate School of Science, Osaka University
Toyonaka, Osaka 560-0043, JAPAN}
\email{yasuda.takehiko.sci@osaka-u.ac.jp}
\begin{abstract}
We prove motivic versions of mass formulas by Krasner, Serre and Bhargava
concerning (weighted) counts of extensions of local fields. 
\end{abstract}

\maketitle
\global\long\def\AA{\mathbb{A}}%
\global\long\def\PP{\mathbb{P}}%
\global\long\def\NN{\mathbb{N}}%
\global\long\def\GG{\mathbb{G}}%
\global\long\def\ZZ{\mathbb{Z}}%
\global\long\def\QQ{\mathbb{Q}}%
\global\long\def\CC{\mathbb{C}}%
\global\long\def\FF{\mathbb{F}}%
\global\long\def\LL{\mathbb{L}}%
\global\long\def\RR{\mathbb{R}}%
\global\long\def\MM{\mathbb{M}}%
\global\long\def\SS{\mathbb{S}}%

\global\long\def\bx{\boldsymbol{x}}%
\global\long\def\by{\boldsymbol{y}}%
\global\long\def\bf{\mathbf{f}}%
\global\long\def\ba{\mathbf{a}}%
\global\long\def\bs{\mathbf{s}}%
\global\long\def\bt{\mathbf{t}}%
\global\long\def\bw{\mathbf{w}}%
\global\long\def\bb{\mathbf{b}}%
\global\long\def\bv{\mathbf{v}}%
\global\long\def\bp{\mathbf{p}}%
\global\long\def\bq{\mathbf{q}}%
\global\long\def\bj{\mathbf{j}}%
\global\long\def\bM{\mathbf{M}}%
\global\long\def\bd{\mathbf{d}}%
\global\long\def\bA{\mathbf{A}}%
\global\long\def\bB{\mathbf{B}}%
\global\long\def\bC{\mathbf{C}}%
\global\long\def\bP{\mathbf{P}}%
\global\long\def\bX{\mathbf{X}}%
\global\long\def\bY{\mathbf{Y}}%
\global\long\def\bZ{\mathbf{Z}}%
\global\long\def\bW{\mathbf{W}}%
\global\long\def\bV{\mathbf{V}}%
\global\long\def\bU{\mathbf{U}}%
\global\long\def\bN{\mathbf{N}}%
\global\long\def\bQ{\mathbf{Q}}%

\global\long\def\cN{\mathcal{N}}%
\global\long\def\cW{\mathcal{W}}%
\global\long\def\cY{\mathcal{Y}}%
\global\long\def\cM{\mathcal{M}}%
\global\long\def\cF{\mathcal{F}}%
\global\long\def\cX{\mathcal{X}}%
\global\long\def\cE{\mathcal{E}}%
\global\long\def\cJ{\mathcal{J}}%
\global\long\def\cO{\mathcal{O}}%
\global\long\def\cD{\mathcal{D}}%
\global\long\def\cZ{\mathcal{Z}}%
\global\long\def\cR{\mathcal{R}}%
\global\long\def\cC{\mathcal{C}}%
\global\long\def\cL{\mathcal{L}}%
\global\long\def\cV{\mathcal{V}}%
\global\long\def\cU{\mathcal{U}}%
\global\long\def\cS{\mathcal{S}}%
\global\long\def\cT{\mathcal{T}}%
\global\long\def\cA{\mathcal{A}}%
\global\long\def\cB{\mathcal{B}}%
\global\long\def\cG{\mathcal{G}}%
\global\long\def\cP{\mathcal{P}}%
\global\long\def\cQ{\mathcal{Q}}%

\global\long\def\fs{\mathfrak{s}}%
\global\long\def\fp{\mathfrak{p}}%
\global\long\def\fm{\mathfrak{m}}%
\global\long\def\fX{\mathfrak{X}}%
\global\long\def\fV{\mathfrak{V}}%
\global\long\def\fx{\mathfrak{x}}%
\global\long\def\fv{\mathfrak{v}}%
\global\long\def\fY{\mathfrak{Y}}%
\global\long\def\fa{\mathfrak{a}}%
\global\long\def\fb{\mathfrak{b}}%
\global\long\def\fc{\mathfrak{c}}%
\global\long\def\fO{\mathfrak{O}}%
\global\long\def\fd{\mathfrak{d}}%
\global\long\def\fP{\mathfrak{P}}%
\global\long\def\fD{\mathfrak{D}}%
\global\long\def\fM{\mathfrak{M}}%
\global\long\def\fC{\mathfrak{C}}%

\global\long\def\rv{\mathbf{\mathrm{v}}}%
\global\long\def\rx{\mathrm{x}}%
\global\long\def\rw{\mathrm{w}}%
\global\long\def\ry{\mathrm{y}}%
\global\long\def\rz{\mathrm{z}}%
\global\long\def\bv{\mathbf{v}}%
\global\long\def\bw{\mathbf{w}}%
\global\long\def\sv{\mathsf{v}}%
\global\long\def\sx{\mathsf{x}}%
\global\long\def\sw{\mathsf{w}}%

\global\long\def\Spec{\mathrm{Spec}\,}%
\global\long\def\Hom{\mathrm{Hom}}%
\global\long\def\Aff{\mathbf{Aff}}%
\global\long\def\ACF{\mathbf{ACF}}%

\global\long\def\Var{\mathbf{Var}}%
\global\long\def\Gal{\mathrm{Gal}}%
\global\long\def\Jac{\mathrm{Jac}}%
\global\long\def\Ker{\mathrm{Ker}}%
\global\long\def\Image{\mathrm{Im}}%
\global\long\def\Aut{\mathrm{Aut}}%
\global\long\def\st{\mathrm{st}}%
\global\long\def\diag{\mathrm{diag}}%
\global\long\def\characteristic{\mathrm{char}}%
\global\long\def\tors{\mathrm{tors}}%
\global\long\def\sing{\mathrm{sing}}%
\global\long\def\red{\mathrm{red}}%
\global\long\def\ord{\operatorname{ord}}%
\global\long\def\pt{\mathrm{pt}}%
\global\long\def\op{\mathrm{op}}%
\global\long\def\Val{\mathrm{Val}}%
\global\long\def\Res{\mathrm{Res}}%
\global\long\def\disc{\mathrm{disc}}%
\global\long\def\Coker{\mathrm{Coker}}%
 
\global\long\def\length{\mathrm{length}}%
\global\long\def\sm{\mathrm{sm}}%
\global\long\def\rank{\mathrm{rank}}%
\global\long\def\et{\mathrm{et}}%
\global\long\def\hom{\mathrm{hom}}%
\global\long\def\tor{\mathrm{tor}}%
\global\long\def\reg{\mathrm{reg}}%
\global\long\def\cont{\mathrm{cont}}%
\global\long\def\Stab{\mathrm{Stab}}%
\global\long\def\GCov{G\textrm{-}\mathrm{Cov}}%
\global\long\def\P{\mathrm{P}}%

\global\long\def\GL{\mathrm{GL}}%
\global\long\def\codim{\mathrm{codim}}%
\global\long\def\prim{\mathrm{prim}}%
\global\long\def\cHom{\mathcal{H}om}%
\global\long\def\cSpec{\mathcal{S}pec}%
\global\long\def\Proj{\mathrm{Proj}\,}%
\global\long\def\modified{\mathrm{mod}}%
\global\long\def\ind{\mathrm{ind}}%
\global\long\def\rad{\mathrm{rad}}%
\global\long\def\Conj{\mathrm{Conj}}%
\global\long\def\Disc{\mathrm{Disc}}%
\global\long\def\hotimes{\hat{\otimes}}%
\global\long\def\Fil{\mathrm{Fil}}%
\global\long\def\Inn{\mathrm{Inn}}%
\global\long\def\rfil{\mathrm{rfil}}%
\global\long\def\per{\mathrm{per}}%
\global\long\def\id{\mathrm{id}}%
\global\long\def\AffVar{\mathbf{AffVar}}%
\global\long\def\Alg{\mathbf{Alg}}%
\global\long\def\PSch{\textit{P-}\mathbf{Sch}}%
\global\long\def\PQVar{\textit{P-}\mathbf{QVar}}%
\global\long\def\Set{\mathbf{Set}}%
\global\long\def\Eis{\mathfrak{Eis}}%
\global\long\def\Sch{\mathbf{Sch}}%
\global\long\def\Shubi{\mathrm{Sh_{ubi}}}%
\global\long\def\unif{\mathrm{unif}}%
\global\long\def\sets{\mathbf{Set}}%
\global\long\def\uni{\mathrm{uni}}%
\global\long\def\Uni{\mathrm{Uni}}%
\global\long\def\sep{\mathrm{sep}}%
\global\long\def\Tr{\operatorname{Tr}}%
\global\long\def\J{\operatorname{J}}%
\global\long\def\tpars{\llparenthesis t\rrparenthesis}%
\global\long\def\tbrats{\llbracket t\rrbracket}%
\global\long\def\Uf{\mathfrak{Uf}}%
\global\long\def\Ut{\mathfrak{Ut}}%
\global\long\def\pr{\mathrm{pr}}%

\section{Introduction}

The aim of this article is to prove motivic versions of mass formulas
by Krasner \cite{krasner1962nombredes,krasner1966nombredes}, Serre
\cite{serre1978unetextquotedblleftformule} and Bhargava \cite{bhargava2007massformulae}.
For a non-archimedean local field $K$ with residue field $\FF_{q}$
and for a positive integer $n$, Serre proved the formula
\begin{equation}
\sum_{L}\frac{q^{-\bd_{L/K}}}{|\Aut(L)|}=q^{1-n},\label{eq:Serre-orig}
\end{equation}
where $L$ runs over isomorphism classes of totally ramified extensions
of $K$ with $[L:K]=n$, $\Aut(L)$ is the group of $K$-automorphisms
and $\bd_{L/K}$ is the discriminant exponent of $L/K$. Using Serre's
formula, Bhargava proved a similar formula
\begin{equation}
\sum_{L}\frac{q^{-\bd_{L/K}}}{|\Aut(L/K)|}=\sum_{j=0}^{n-1}P(n,n-j)q^{-j},\label{eq:Bha-orig}
\end{equation}
this time $L$ running over isomorphism classes of étale $K$-algebras
of degree $n$. Here $P(n,i)$ denotes the number of partitions of
$n$ into exactly $i$ positive integers. 

In \cite{yasuda2017thewild}, the author proved these formulas by
a different approach which was based upon observation from \cite{wood2015massformulas}
which relates Bhargava's formula with the Hilbert scheme of points.
In this approach, we can prove Bhargava's formula first and deduce
Serre's formula from it, using relation between the two formulas via
the exponential formula which was observed by Kedlaya \cite{kedlaya2007massformulas}. 

Before Serre obtained his formula, Krasner had obtained a formula
for the number of degree-$n$ extensions of $K$ with a prescribed
discriminant exponent as well as one restricted to totally ramified
extensions. Thus, Krasner's result is a refinement of Serre's one
and we can derive Serre's from Krasner's. Under the condition that
$K$ has characteristic $p>0$, the situation that we focus on in
the present paper, the most interesting case is when $p\mid n$, $m-n+1>0$
and $p\nmid(m-n+1)$. Under these conditions, the number of degree-$n$
totally ramified extensions of $K$ in a fixed algebraic closure of
$K$ with discriminant exponent $m$ is 
\[
n(q-1)q^{\lfloor(m-n+1)/p\rfloor}.
\]
We may rewrite this formula as
\begin{equation}
\sum_{L}\frac{1}{|\Aut(L)|}=(q-1)q^{\lfloor(m-n+1)/p\rfloor},\label{eq:Kra}
\end{equation}
where $L$ runs over isomorphism classes of such extensions of $K$(instead
of counting subfields of $\overline{K}$). 

To formulate motivic versions of these formulas, we consider the P-moduli
space $\Delta_{n}$ (resp. $\Delta_{n}^{\circ}$) of degree-$n$ étale
covers (resp. connected covers) of the punctured formal disk $\Spec k\llparenthesis t\rrparenthesis$
with $k$ denoting a field, constructed in \cite{tonini2023moduliof}.
The notion of P-moduli space is even coarser than the one of coarse
moduli space, but enough to define motivic integrals that we consider
below. For details, see Section (\ref{sec:Strong-P-moduli-spaces}).
The discriminant exponent defines a constructible function $\bd\colon\Delta_{n}\longrightarrow\ZZ$
as well as its restriction to $\Delta_{n}^{\circ}$. We can define
the integral 
\[
\int_{\Delta_{n}}\LL^{-\bd}:=\sum_{m=0}^{\infty}[\bd^{-1}(m)]\LL^{-m}
\]
and similarly the integral $\int_{\Delta_{n}^{\circ}}\LL^{-\bd}$
in a version of the complete Grothendieck ring of varieties, denoted
by $\widehat{\cM}_{k}^{\heartsuit}$ (for the definition of this ring,
see Definition \ref{def:Grothendieck ring}). Motivic versions of
formulas (\ref{eq:Serre-orig}) and (\ref{eq:Bha-orig}) by Serre
and Bhargava are formulated as follows:
\begin{thm}[Theorem \ref{thm:motivic-Serre} and Corollary \ref{cor:motivic-Bhargava}]
\label{thm:main2}We have the following equalities in $\widehat{\cM}_{k}^{\heartsuit}$:
\begin{gather*}
\int_{\Delta_{n}^{\circ}}\LL^{-\bd}=\LL^{1-n},\\
\int_{\Delta_{n}}\LL^{-\bd}=\sum_{j=0}^{n-1}P(n,n-j)\LL^{-j}.
\end{gather*}
\end{thm}

The second equality of the theorem is equivalent to \cite[Corollary 1.5]{yasuda2024motivic}
via the correspondence between the discriminant exponent and the Artin
conductor \cite{wood2015massformulas}, except a slight difference
of the used Grothendieck rings. The proof in \cite{yasuda2024motivic}
is by translating the proof in \cite{yasuda2017thewild} of Bhargava's
formula into the motivic setting. 

We also prove a motivic version of Krasner's formula (\ref{eq:Kra}).
Let $\Delta_{n}^{(m)}:=\bd^{-1}(m)\subset\Delta_{n}^{\circ}$, the
locus of étale covers with discriminant exponent $m$. This is a constructible
subset. 
\begin{thm}[see Theorem \ref{thm:motivic-Krasner} for the full statement of the
result.]
If $k$ has characteristic $p>0$ and if $p\mid n$, $m-n+1>0$ and
$p\nmid(m-n+1)$, then we have
\[
[\Delta_{n}^{(m)}]=(\LL-1)\LL^{\lfloor(m-n+1)/p\rfloor}
\]
in $\widehat{\cM}_{k}^{\heartsuit}$. 
\end{thm}

The author used this result as a working hypothesis in a previous
paper \cite[Section 12]{yasuda2016wildermckay} to verify a certain
duality in mass formulas. He also applies this theorem to a study
of quotient singularities in another paper in preparation \cite{yasudaquotient}. 

We now explain the outline of the proof of Theorem \ref{thm:main2}.
Our strategy is to translate Serre's arguments by using $p$-adic
measures to the motivic setting by using the theory of motivic integration,
a theory pioneered by Kontsevich \cite{kontsevich1995lecture} and
Denef--Loeser \cite{denef1999germsof}. We use a version of the theory
over a complete discrete valuation ring established by Sebag \cite{sebag2004integration}.
We consider the space of Eisenstein polynomials of degree $n$ with
coefficients in $k\llbracket t\rrbracket$, denoted by $\Eis$. We
regard this as a subspace of the arc space $\J_{\infty}(\AA_{k\llbracket t\rrbracket}^{n})$
of the affine space $\AA_{k\llbracket t\rrbracket}^{n}$. Through
the natural map $\Eis\to\Delta_{n}^{\circ}$, we relate the integral
$\int_{\Delta_{n}^{\circ}}\LL^{-\bd}$ with the motivic volume of
$\Eis/\GG_{m}$, the quotient of $\Eis$ by a natural action of $\GG_{m}$.
This motivic volume is easy to compute, leading to the motivic version
of Serre's formula. The motivic version of Bhargava's formula easily
follows from Serre's one. The motivic version of Krasner's formula
is obtained by an explicit description of the locus of Eisenstein
polynomials giving extensions of a prescribed discriminant exponent. 

Throughout the paper, we work over a field $k$ of characteristic
$p\ge0$. A $k$\emph{-variety} means a separated scheme of finite
type over $k$. We follow the convention that when $p=0$, then $p$
is coprime to any positive integer $n$ and we write $p\nmid n$.
The symbol $K$ means an extension of $k$, unless otherwise noted.
For a ring $R$, we denote by $R\llbracket t\rrbracket$ the ring
of power series with coefficients in $R$ and by $R\llparenthesis t\rrparenthesis$
its localization by $t$, which is nothing but the ring of Laurent
power series with coefficients in $R$. When $K$ is a field and $A$
is an étale $K\llparenthesis t\rrparenthesis$-algebra, we denote
by $\cO_{A}$ the integral closure of $K\llbracket t\rrbracket$ in
$A$. For a finite group $G$, a \emph{$G$-torsor} means an étale
$G$-torsor. 

\subsection*{Acknowledgements}

I wrote the first draft of the part concerning the motivic versions
of formulas by Serre and Bhargava around the autumn of 2018, intending
to contain it in my joint paper with Fabio Tonini \cite{tonini2023moduliof}.
But we decided not to contain it for reasons of time and length of
the paper. In the summer of 2024, I completed the manuscript in order
to publish it as a separate paper, containing an extra result, the
motivic version of Krasner's formula. I deeply thank Fabio Tonini
for permitting me to do this and for stimulating discussions during
our joint work. I also thank Melanie Matchett Wood for helpful discussions
during our joint works, which resulted in papers \cite{wood2015massformulas,wood2017massformulas}. 

\section{Discriminants\label{sec:Discriminants}}

Let $R$ be a ring and let $S$ be an $R$-algebra which is free of
rank $n$ as an $R$-module. For each element $s\in S$, the map $S\to S,\,x\mapsto sx$
is an $R$-linear map and its trace $\Tr(s)$ is defined as an element
of $R$. The \emph{discriminant} $D(s_{1},\dots,s_{n})$ of an $R$-module
basis $s_{1},\dots,s_{n}\in S$ is defined to be the determinant of
the $n\times n$ matrix $(\Tr(s_{i}s_{j}))_{i,j}$ with entries in
$R$. It is known that the ideal generated by $D(s_{1},\dots,s_{n})$
is independent of the choice of basis and that $S$ is étale over
$R$ if and only if $(D(s_{1},\dots,s_{n}))=R$ or equivalently $D(s_{1},\dots,s_{n})$
is an invertible element of $R$. 

Consider the case $R=\ZZ[Y_{1},\dots,Y_{n}]$, the $n$-variate polynomial
ring with integer coefficients, and 
\[
S=R[x]/(x^{n}+Y_{1}x^{n-1}+\cdots+Y_{n-1}x+Y_{n}).
\]
We define the \emph{discriminant polynomial} $F(Y_{1},\dots,Y_{n})\in R$
to be the discriminant $D(1,x,\dots,x^{n-1})$ of the $R$-basis $1,x,\dots,x^{n-1}$
of $S$. For any ring $R'$ and an $R'$-algebra 
\[
S'=R[x]/(x^{n}+y_{1}x^{n-1}+\cdots+y_{n-1}x+y_{n})\quad(y_{1},\dots,y_{n}\in R),
\]
we have $D(1,x_{S'},\dots,x_{S'}^{n-1})=F(y_{1},\dots,y_{n})$, where
$x_{S'}$ is the image of $x$ in $S'$. The $R'$-algebra $S'$ is
étale if and only if $F(y_{1},\dots,y_{n})$ is an invertible element
of $R'$. 

Next consider the case where $R=K\llparenthesis t\rrparenthesis$
with $K$ a field and 
\[
S=R[x]/(x^{n}+y_{1}x^{n-1}+\cdots+y_{n-1}x+y_{n})\quad(y_{1},\dots,y_{n}\in R).
\]
The polynomial in the last equality is called an \emph{Eisenstein
polynomial }if $\ord y_{i}>0$ for every $i$ and $\ord y_{n}=1$.
If this is the case, from \cite[p.~19]{serre1979localfields}, $S$
is a discrete valuation field with uniformizer $x$ and has residue
field isomorphic to $K$. The field extension $S/R$ is separable
if and only if $F(y_{1},\dots,y_{n})\ne0$. When $S/R$ is separable,
its \emph{discriminant exponent} is defined to be 
\[
\bd_{S/R}:=\ord F(y_{1},\dots,y_{n})\in\ZZ_{\ge0}.
\]
We often write $\bd_{S/R}$ simply as $\bd_{S}$, omitting $R$.

\section{Strong P-moduli spaces\label{sec:Strong-P-moduli-spaces}}

In this section, we recall results from \cite{tonini2023moduliof}.
Let $\Aff/k$ be the category of affine schemes over $k$ and let
$\ACF/k$ be its full subcategories consisting of spectra $\Spec K$
with $K$ an algebraically closed field. We identify a $k$-scheme
$X$ with the associated functor 
\[
(\Aff/k)^{\op}\to\Set,\,T\mapsto\Hom_{k}(T,X).
\]
For a functor $Z\colon(\Aff/k)^{\op}\to\Set$, we let $Z_{F}\colon(\ACF/k)^{\op}\to\Set$
be its restriction to $(\ACF/k)^{\op}$. 
\begin{defn}
Let $Y$ be a $k$-scheme and let $X$ be a functor $(\Aff/k)^{\op}\to\Set$
(e.g.~a $k$-scheme). A \emph{P-morphism} $f\colon Y\to X$ is a
natural transformation $Y_{F}\to X_{F}$ such that there exist morphisms
of $h\colon Z\to Y$ and $g\colon Z\to X$ of $k$-schemes such that
$h$ is surjective and locally of finite type and the following diagram
is commutative:
\begin{equation}
\xymatrix{Z_{F}\ar[d]_{h_{F}}\ar[dr]^{g_{F}}\\
Y_{F}\ar[r]_{f} & X_{F}
}
\label{eq:P-dia}
\end{equation}
We denote by $\Hom_{k}^{P}(Y,X)$ the set of P-morphisms over $k$
from $Y$ to $X$. We denote by $X^{P}$ the functor 
\[
(\Aff/k)^{\op}\to\Set,\,T\mapsto\Hom_{k}^{P}(T,X).
\]
A P-morphism $f\colon Y\to X$ is said to be a \emph{P-isomorphism
}if there exists a P-morphism $g\colon X\to Y$ such that both $f\circ g$
and $g\circ f$ are the identities. 
\end{defn}

For a scheme $X$, we denote its underlying point set by $|X|$. This
set is identified with the set of equivalence classes of geometric
points $\Spec K\to X$; two geometric points $\Spec K\to X$ and $\Spec K'\to X$
are equivalent if they fit into the commutative diagram
\[
\xymatrix{\Spec K''\ar[r]\ar[d] & \Spec K'\ar[d]\\
\Spec K\ar[r] & X
}
\]
with $K''$ also an algebraically closed field. The set $|X|$ is
equipped with the Zariski topology. A P-morphism $f\colon Y\to X$
induces a map $|f|\colon|Y|\to|X|$ in the obvious way. 
\begin{lem}[{\cite[Lemmas 4.7 and 4.32]{tonini2023moduliof} }]
Let $Y$ and $X$ be separated schemes locally of finite type over
$k$. Let $f\colon Y\to X$ be a P-morphism.
\begin{enumerate}
\item There exist geometrically bijective and finite-type morphisms $h\colon Z\to Y$
and $g\colon Z\to X$ which make diagram (\ref{eq:P-dia}) commutative. 
\item If $\Gamma_{f}\colon Y\to Y\times_{k}X$ is the graph of $f$, then
$\Image(|\Gamma_{f}|)$ is a locally constructible subset of $|Y\times_{k}X|$. 
\item $f$ is a P-isomorphism if and only if it is geometrically bijective
(that is, for every algebraically closed field $K$, $f(K)\colon Y(K)\to X(K)$
is bijective.)
\end{enumerate}
\end{lem}

\begin{defn}
Keeping the assumption of the lemma, we denote the locally constructible
subset $\Image(|\Gamma_{f}|)$ again by $\Gamma_{f}$. For a point
$x\colon\Spec K\to X$ with $K$ any field, the \emph{fiber} $f^{-1}(x)$
is defined to be the constructible subset 
\[
(\Gamma_{f}\times_{X,x}\Spec K)\cap\pr_{X}^{-1}(x)\subset Y\otimes_{k}K.
\]
Here $\Gamma_{f}\times_{X,x}\Spec K$ means the preimage of $\Gamma_{f}$
by the morphism
\[
\id\times x\colon Y\otimes_{k}K\to Y\times_{k}X.
\]
\end{defn}

\begin{defn}
We define the \emph{category of P-schemes over $k$}, denoted by $\PSch/k$,
to be the categories having $k$-schemes as objects and P-morphisms
over $k$ as morphisms. 
\end{defn}

\begin{defn}
Let $\cF\colon(\Aff/k)^{\op}\to\Set$. A \emph{strong P-moduli space}
of of $\cF$ is a $k$-scheme $X$ given with a morphism $\pi\colon\cF\to X^{P}$
such that the induced morphism $\pi\colon\cF^{P}\to X^{P}$ is an
isomorphism.
\end{defn}

If exists, a strong P-moduli space is unique up to a unique P-isomorphism.
By definition, if $X$ is a strong P-moduli space of $\cF$, then
for every algebraically closed field $K$, the map $\cF(K)\to X(K)$
is bijective.
\begin{thm}[{\cite[Theorem 8.9]{tonini2023moduliof}}]
\label{thm:P-moduli}For a finite group $G$, the functor 
\[
\cF_{G}\colon(\Aff/k)^{\op}\to\Set,\,\Spec R\mapsto\{G\textrm{-torsors over \ensuremath{\Spec R\tpars}}\}/{\cong}
\]
has a strong P-moduli space which is a countable coproduct of affine
$k$-varieties. Here a $G$-torsor means an étale $G$-torsor. 
\end{thm}

\begin{defn}
For a ring $R$, we say that a finite étale $R\tpars$-algebra $A$
is of \emph{degree $n$} (resp.~\emph{of discriminant exponent $m$},
\emph{connected}) if for every point $\Spec K\to\Spec R$, the induced
$K\tpars$-algebra $A\otimes_{R\tpars}K\tpars$ is of degree $n$
(resp.~of discriminant exponent $m$, a field). For $n\in\ZZ_{>0}$
and $m\in\ZZ_{\ge0}$, we define the following functors:
\begin{align*}
\cF_{n}\colon(\Aff/k)^{\op} & \to\Set\\
\Spec R & \mapsto\{\textrm{finite étale \ensuremath{R\tpars}-algebras of degree \ensuremath{n}}\}/{\cong},
\end{align*}
\begin{align*}
\cF_{n}^{\circ}\colon(\Aff/k)^{\op} & \to\Set\\
\Spec R & \mapsto\{\textrm{connected finite étale \ensuremath{R\tpars}-algebras of degree \ensuremath{n}}\}/{\cong}
\end{align*}
and 
\begin{align*}
\cF_{n}^{(m)}\colon(\Aff/k)^{\op} & \to\Set\\
\Spec R & \mapsto\left\{ \begin{gathered}\textrm{connected finite étale \ensuremath{R\tpars}-algebras of}\\
\textrm{degree \ensuremath{n} and discriminant exponent \ensuremath{m}}
\end{gathered}
\right\} /{\cong}.
\end{align*}
\end{defn}

\begin{cor}
The functor $\cF_{n}$, $\cF_{n}^{\circ}$ and $\cF_{n}^{(m)}$ have
strong P-moduli spaces which are coproducts of countably many affine
$k$-varieties.
\end{cor}

\begin{proof}
Since $\cF_{n}$ is isomorphic to $\cF_{S_{n}}$, Theorem \ref{thm:P-moduli}
shows that $\cF_{n}$ has a strong P-moduli space which is a coproduct
of countably many affine $k$-varieties. From \cite[Lemma 8.7 and Theorem 8.9]{tonini2023moduliof},
$\cF_{n}^{\circ}$ also has a strong P-moduli space which is a coproduct
of countably many affine $k$-varieties. From \cite{wood2015massformulas},
the discriminant exponent function $\bd\colon\cF_{n}\to\ZZ_{\ge0}$,
which is identified with the Artin conductor $\ba\colon\cF_{S_{n}}\to\ZZ_{\ge0}$,
is a special case of $v$-function $\bv\colon\cF_{S_{n}}\to\ZZ_{\ge0}$.
From \cite[Theorem 9.8]{tonini2023moduliof}, $\bd$ is a locally
constructible function. Hence, the property ``$\bd=m$'' is locally
constructible. From \cite[Theorem 8.9]{tonini2023moduliof}, $\cF_{n}^{(m)}$
also has a P-moduli spaces which is a coproduct of countably many
affine $k$-varieties.
\end{proof}
\begin{defn}
We denote by $\Delta_{n}^{\circ}$ and $\Delta_{n}^{(m)}$ P-moduli
spaces of $\cF_{n}^{\circ}$ and $\cF_{n}^{(m)}$, which are coproducts
of countably many affine $k$-varieties.
\end{defn}

By the definition of strong P-moduli space, for each connected finite
étale $R\tpars$-algebra of degree $n$ and discriminant exponent
$m$, we have the induced P-morphism $\Spec R\to\Delta_{n}^{(m)}$.
\begin{rem}
\label{rem:const-subsets}Let us write $\Delta_{n}^{\circ}=\coprod_{i\in I}W_{i}$,
where $I$ is a countable set and $W_{i}$ are $k$-varieties. Then,
locally constructible subsets and constructible subsets of $\Delta_{n}^{\circ}$
are characterized as follows. A subset $C\subset\Delta_{n}^{\circ}$
is locally constructible if and only if for every $i$, $C\cap W_{i}$
is a constructible subset of $W_{i}$. A locally constructible subset
$C\subset\Delta_{n}^{\circ}$ is constructible if it is quasi-compact
or equivalently if it is contained in $\bigcup_{i\in I_{0}}W_{i}$
for a finite subset $I_{0}\subset I$. Note that whether a subset
$C\subset\Delta_{n}^{\circ}$ is locally constructible (resp.~constructible)
is independent of the choice of P-moduli space: if $(\Delta_{n}^{\circ})'$
is another P-moduli space of $\cF_{n}^{\circ}$ and $C'\subset(\Delta_{n}^{\circ})'$
is the subset corresponding to $C$, then $C$ is locally constructible
(resp.~constructible) if and only if so is $C'$.
\end{rem}

\section{The space of Eisenstein polynomials}

In this section, we construct the space of Eisenstein polynomials
as a subspace of an arc space and study its properties. We refer the
reader to \cite{chambert-loir2018motivic} for details on arc spaces,
in particular, from the viewpoint of motivic integration. 

Let $V:=\AA_{k\llbracket t\rrbracket}^{n}=\Spec k\llbracket t\rrbracket[x_{1},\dots,x_{n}]$
and let $\J_{\infty}V$ and $\J_{l}V$, $l\in\ZZ_{\ge0}$, be its
arc scheme and jet schemes. Namely, for a $k$-algebra $R$, we have
\begin{align*}
(\J_{\infty}V)(R) & =\Hom_{k\llbracket t\rrbracket}(\Spec R\llbracket t\rrbracket,V),\\
(\J_{l}V)(R) & =\Hom_{k\llbracket t\rrbracket}(\Spec R\llbracket t\rrbracket/(t^{l+1}),V).
\end{align*}
For $l',l\in\ZZ_{\ge0}$ with $l'\ge l$, we have truncation morphisms
\[
\pi_{l}\colon\J_{\infty}V\to\J_{l}V\text{ and }\pi_{l}^{l'}\colon\J_{l'}V\to\J_{l}V.
\]
 
\begin{defn}
For a $k$-algebra $R$, $R$-points of $\J_{\infty}V$ correspond
to $n$-tuples of power series $y=(y_{1},\dots,y_{n})\in R\llbracket t\rrbracket^{n}$.
We let them correspond also to polynomials 
\[
f_{y}(x):=x^{n}+y_{1}x^{n-1}+\cdots+y_{n-1}x+y_{n}\in R\llbracket t\rrbracket[x].
\]
We let $A_{y}$ be the $R\llparenthesis t\rrparenthesis$-algebra
$R\llparenthesis t\rrparenthesis[x]/(f_{y}(x))$. 
\end{defn}

As we saw in Section \ref{sec:Discriminants}, the extension $A_{y}/L\llparenthesis t\rrparenthesis$
is separable if and only if $F(y)\ne0$. 

\begin{defn}
For indeterminates $Y_{i,j}$ ($1\le i\le n$, $j\ge0$) and for integers
$m\ge0$, we define polynomials $F_{m}(Y_{i,j})\in\ZZ[Y_{i,j};j\le m]$
by 
\[
F\left(\sum_{j\ge0}Y_{1,j}t^{j},\dots,\sum_{j\ge0}Y_{n,j}t^{j}\right)=\sum_{m\ge0}F_{m}(Y_{i,j})t^{m}.
\]
\end{defn}

\begin{defn}
Let $\Eis\subset\J_{\infty}V$ (resp. $\Eis^{\sep},\Eis^{(m)}$) to
be the locus of Eisenstein polynomials (resp. separable Eisenstein
polynomials, Eisenstein polynomials whose discriminant have order
$m$). 
\end{defn}

If we write $y_{i}=\sum_{j\in\ZZ_{\ge0}}y_{i,j}t^{j}$, then the above
loci are described as
\begin{align*}
\Eis & =\{(y_{i,j})\mid\text{for every \ensuremath{i}, }y_{i,0}=0\text{ and }y_{n,1}\ne0\},\\
\Eis^{\sep} & =\{(y_{i,j})\in\Eis\mid\text{for some \ensuremath{m}, }F_{m}(y_{i,j})\ne0\},\\
\Eis^{(m)} & =\{(y_{i,j})\in\Eis\mid\text{for \ensuremath{m'\le m}, }F_{m'}(y_{i,j})=0\text{ and }F_{m}(y_{i,j})\ne0\}.
\end{align*}
These are locally closed subsets of $\J_{\infty}V$. We have $\Eis^{\sep}=\bigsqcup_{m\ge0}\Eis^{(m)}$.
For $y\in\Eis^{\sep}(K)$, the associated extension $A_{y}/K\tpars$
is separable and totally ramified. Let $\varpi\in A_{y}=K\llparenthesis t\rrparenthesis[x]/(f_{y}(x))$
denote the image of $x$, which is a uniformizer of $A_{y}$. As is
well-known, the discriminant exponent $\bd_{A_{y}/K\tpars}$ of $A_{y}/K\tpars$
is also equal to
\begin{equation}
n\ord f'(\varpi)=n\ord\left(\sum_{i=0}^{n}(n-i)y_{i}\varpi^{n-i-1}\right)\label{eq:disc}
\end{equation}
with $y_{0}=1$ (for example, see \cite[Proposition 6 on p. 50 and  Corollary 2 on p. 56]{serre1979localfields}).
Here we denote the unique extension of the valuation $\ord\colon K\tpars\to\ZZ_{\ge0}\cup\{\infty\}$
to $A_{y}$ again by $\ord$. This equality shows that if $p\nmid n$,
then $\bd_{A_{y}/L\tpars}=n-1$ and hence $\Eis^{(m)}=\emptyset$
for $m\ne n-1$. For $n$ with $p\nmid n$, we have the following
explicit description of $\Eis^{(m)}$. 

\begin{prop}
\label{prop:Eis-m-explicit}Suppose that $p>0$ and $p\mid n$. Let
\[
y=(y_{1},\dots,y_{n})\in K\tbrats^{n}=\Eis(K).
\]
We have $y\in\Eis^{(m)}(K)$ if and only if for every $i\in\{1,\dots,n-1\}$
with $p\nmid(n-i)$, the inequality 
\[
\ord y_{i}\ge\left\lceil \frac{m-n+1+i}{n}\right\rceil 
\]
holds and the equality in this inequality holds if $m+i+1\in n\ZZ$.
\end{prop}

\begin{proof}
In the situation of the proposition, we have
\[
\bd_{A_{y}/K\tpars}=n\ord\left(\sum_{\substack{0\le i<n\\
p\nmid(n-i)
}
}(n-i)y_{i}\varpi^{n-i-1}\right).
\]
For $i$ with $p\nmid(n-i)$, we have
\begin{align*}
n\ord\left((n-i)y_{i}\varpi^{n-i-1}\right) & \equiv n-i-1\mod n\ZZ.
\end{align*}
In particular, for distinct $i$'s, these values are different to
one another. Therefore,
\[
n\ord f'(\varpi)=\min\left\{ n-i-1+n\ord y_{i}\mid0\le i<n,\,p\nmid(n-i)\right\} .
\]
Moreover, if $\bd_{A_{y}/K\tpars}=m$, the minimum is attained at
$i$ with $m+i+1\in n\ZZ$. This shows the lemma. 
\end{proof}
\begin{defn}
For $l\in\ZZ_{\ge0}$, we let $\Eis_{l}^{(m)}\subset\J_{l}V$ be the
image of $\Eis^{(m)}$ by $\pi_{l}$. 
\end{defn}

\begin{cor}
\label{cor:pi_l(Eis^m)}Let $c:=m-n+1$. Suppose that $p\mid n$,
$p\nmid c$ and $c\ge0$. For $l\ge\lfloor(c+n-1)/n\rfloor=\lfloor m/n\rfloor$,
$\Eis_{l}^{(m)}$ is a locally closed subset of $\J_{l}V$. Moreover,
if we give it the reduced scheme structure, then
\[
\Eis_{l}^{(m)}\cong\GG_{m}^{2}\times\AA_{k}^{nl-c+\lfloor c/p\rfloor-1}.
\]
In particular, $\Eis_{l}^{(m)}$ is an affine variety. 
\end{cor}

\begin{proof}
The jet scheme $\J_{l}V$ is the affine space $\AA_{k}^{n(l+1)}$
with the coordinates $y_{i,j}$, $1\le i\le n$, $0\le j\le l$. The
subset $\Eis_{l}^{(m)}$ of it is defined by 
\[
\begin{cases}
y_{i,0}=0 & (1\le i\le n),\\
y_{n,1}\ne0,\\
y_{i,j}=0 & \left(i<n,\,p\nmid(n-i),\,j<\left\lceil \frac{c+i}{n}\right\rceil \right),\\
y_{i,\frac{c+i}{n}}\ne0 & \left(i<n,\,p\nmid(n-i),\,c+i\in n\ZZ\right),
\end{cases}
\]
which shows that $\Eis_{l}^{(m)}$ is a locally closed subset. The
last two conditions can be rephrased as:
\[
\begin{cases}
y_{i,j}=0 & \left(i<n,\,p\nmid(n-i),\,c+i\notin n\ZZ,\,j\le\left\lfloor \frac{c+i}{n}\right\rfloor \right)\\
y_{i,j}=0 & \left(i<n,\,p\nmid(n-i),\,c+i\in n\ZZ,\,j<\left\lfloor \frac{c+i}{n}\right\rfloor \right)\\
y_{i,\frac{c+i}{n}}\ne0 & \left(i<n,\,p\nmid(n-i),\,c+i\in n\ZZ\right).
\end{cases}
\]
Thus, we have 
\[
\Eis_{l}^{(m)}\cong\GG_{m}^{2}\times\AA_{k}^{nl-s-1},
\]
where 
\[
s=\sum_{\substack{1\le i<n\\
p\nmid(n-i)
}
}\left\lfloor \frac{c+i}{n}\right\rfloor .
\]
Let us write $n=pn'$. From Hermite's identity (for example, see \cite[Chapter 12]{svetoslavsavchev2003mathematical}),
we have 
\begin{align*}
s & =\sum_{i=1}^{n-1}\left\lfloor \frac{c}{n}+\frac{i}{n}\right\rfloor -\sum_{i=1}^{n'-1}\left\lfloor \frac{c}{n}+\frac{i}{n'}\right\rfloor \\
 & =\left\lfloor n\frac{c}{n}\right\rfloor -\left\lfloor n'\frac{c}{n}\right\rfloor \\
 & =c-\left\lfloor \frac{c}{p}\right\rfloor .
\end{align*}
\end{proof}
\begin{lem}
For a $k$-algebra $R$ and for a point $y\in\Eis^{(m)}(R)$, $A_{y}$
is étale over $R\llparenthesis t\rrparenthesis$. 
\end{lem}

\begin{proof}
Since the leading coefficient $F_{m}(y_{i,j})$ of $F(y_{1},\dots,y_{n})$
is an invertible element of $R$, $F(y)\in R\llparenthesis t\rrparenthesis$
is invertible and $A_{y}/R\llparenthesis t\rrparenthesis$ is étale.
\end{proof}
Since $\Eis^{(m)}$ itself is an affine scheme having the coordinate
ring 
\[
S=k[y_{i,j},y_{n,1}^{-1},F_{m}^{-1}\mid i\in\{1,\dots,n\},j\in\ZZ_{\ge0}]/(F_{m'}\mid m'\le m),
\]
we have the corresponding étale algebra over $S\tpars$ and the induced
P-morphism
\begin{gather*}
\psi^{(m)}\colon\Eis^{(m)}\to\Delta_{n}^{(m)}.
\end{gather*}
For each point $y\in\Eis^{(m)}(R)$, the composition P-morphism
\[
\Spec R\xrightarrow{y}\Eis^{(m)}\to\Delta_{n}^{(m)}
\]
is the P-morphism associated to the $R\tpars$-algebra $A_{y}$. 
\begin{rem}
For $y\in\Eis(R)$, even if $K\llparenthesis t\rrparenthesis\to K\llparenthesis t\rrparenthesis[x]/(f_{y})$
is étale for every point $\Spec K\to\Spec R$ with $K$ a field, the
map $R\llparenthesis t\rrparenthesis\to R\llparenthesis t\rrparenthesis[x]/(f_{y})$
is not generally étale as the following example exists. Suppose that
$k$ has characteristic two. Let $R=k[s]$ and $f=x^{2}+(st+(s+1)t^{2})x+t\in R\llparenthesis t\rrparenthesis[x]$.
Then $R\llparenthesis t\rrparenthesis[x]/(f)$ is not étale over $R\llparenthesis t\rrparenthesis$.
Indeed, $df/dx=st+(s+1)t^{2}$ is not a unit in $R\llparenthesis t\rrparenthesis$.
Hence it is not a unit in $R\llparenthesis t\rrparenthesis[x]/(f)$
either. On the other hand, for any point $\Spec K\to\Spec R$, since
$df/dx$ is a unit in $K\llparenthesis t\rrparenthesis$, the induced
map $K\llparenthesis t\rrparenthesis\to K\llparenthesis t\rrparenthesis[x]/(f)$
is étale. This explains why we need to decompose $\Eis^{\sep}$ into
subsets $\Eis^{(m)}$ to have a map to $\Delta_{n}^{\circ}$.
\end{rem}

For $y\in\Eis_{l}^{(m)}(R)\subset(R\llbracket t\rrbracket/(t^{l+1}))^{n}$,
let $\tilde{y}\in\Eis^{(m)}(R)\subset R\llbracket t\rrbracket^{n}$
be its canonical lift given by 
\[
\tilde{y}_{i,j}=\begin{cases}
y_{i,j} & (j\le m)\\
0 & (j>m).
\end{cases}
\]
For $l\ge\lfloor m/n\rfloor$, the assignment $y\mapsto\widetilde{y}$
defines a morphism $\Eis_{l}^{(m)}\to\Eis^{(m)}$, which is a section
of $\pi_{l}|_{\Eis^{(m)}}\colon\Eis^{(m)}\to\Eis_{l}^{(m)}$. We define
the morphism 
\[
\psi_{l}=\psi_{l}^{(m)}\colon\Eis_{l}^{(m)}\to\Delta_{n}^{\circ},\,y\mapsto A_{\tilde{y}},
\]
which is the composition of the section $\Eis_{l}^{(m)}\to\Eis^{(m)}$
and $\psi\colon\Eis^{(m)}\to\Delta_{n}^{(m)}$. 

We need the following lemma, which is a variant of Fontaine's \cite[Prop. ~1.5]{fontaine1985ilny}.
\begin{lem}
\label{lem:Fontaine-variant}Let $L/K\llparenthesis t\rrparenthesis$
be a finite separable extension and $E/K\llparenthesis t\rrparenthesis$
any algebraic extension. Let $\cO_{L}$ and $\cO_{E}$ be the integral
closures of $K\tbrats$ in $L$ and $E$, respectively. Let $l$ be
an integer with  $l>\bd_{L/K\llparenthesis t\rrparenthesis}$. Suppose
that there exists a $K\llbracket t\rrbracket$-algebra homomorphism
$\eta\colon\cO_{L}\to\cO_{E}/t^{l}\cO_{E}$. Then there exists a $K\llparenthesis t\rrparenthesis$-embedding
$L\to E$.
\end{lem}

\begin{proof}
The proof is also similar to the one of \cite[Proposition 1.5]{fontaine1985ilny}.
Let us write $\cO_{L}=K\llbracket t\rrbracket[\alpha]$ and let $f(X)$
be the minimal polynomial of $\alpha$ over $K\llparenthesis t\rrparenthesis$,
which is of degree $n=[L:K\llparenthesis t\rrparenthesis]$. We embed
$L$ and $E$ into an algebraic closure $\Omega$ of $K\llparenthesis t\rrparenthesis$
and denote the extension of the valuation $\ord$ to $\Omega$ again
by $\ord$. Let $\beta\in\cO_{E}$ be a lift of $\eta(\alpha)$. Then
$\ord\,f(\beta)\ge l$. Let $\alpha=\alpha_{1},,\dots,\alpha_{n}\in\Omega$
be the conjugates of $\alpha$. Since $f(\beta)=\prod_{i=1}^{n}(\beta-\alpha_{i})$,
we have
\begin{equation}
\sup_{i}\ord(\beta-\alpha_{i})\ge\frac{\ord(f(\beta))}{n}\ge\frac{l}{n}>\frac{\bd_{L}}{n}.\label{eq:ineq1}
\end{equation}
Recall that the discriminant of $L/K\llparenthesis t\rrparenthesis$
is the ideal generated by $\prod_{i\ne j}(\alpha_{j}-\alpha_{i})$.
Suppose that $\ord(\alpha_{1}-\alpha_{2})=\sup_{i\ne j}\ord(\alpha_{i}-\alpha_{j})$.
If $\sigma_{i}$, $1\le i\le n$, are $K\tpars$-automorphisms of
$\Omega$ with $\sigma_{i}(\alpha)=\alpha_{i}$. Then, 
\begin{equation}
\bd_{L}\ge\sum_{i\ne j}\ord(\alpha_{i}-\alpha_{j})\ge\sum_{i=1}^{n}\ord(\sigma_{i}(\alpha_{1})-\sigma_{i}(\alpha_{2}))\ge n\ord(\alpha_{1}-\alpha_{2}).\label{eq:ineq2}
\end{equation}
Combining (\ref{eq:ineq1}) and (\ref{eq:ineq2}) gives
\[
\sup_{i}\ord(\beta-\alpha_{i})>\sup_{i\ne j}\ord(\alpha_{i}-\alpha_{j}).
\]
From Krasner's lemma \cite[II,  S2,  Proposition3]{lang1994algebraic},
for $i$ with $\ord(\beta-\alpha_{i})=\sup_{i}\ord(\beta-\alpha_{i})$,
we have $K\llparenthesis t\rrparenthesis(\alpha_{i})\subset K\llparenthesis t\rrparenthesis(\beta)\subset E$.
The composite map 
\[
L\xrightarrow{\sim}K\llparenthesis t\rrparenthesis(\alpha_{i})\hookrightarrow K\llparenthesis t\rrparenthesis(\beta)\hookrightarrow E
\]
is a $K\tpars$-embedding.
\end{proof}
\begin{cor}
\label{cor:Fontaine-var}Let $L$ and $L'$ be finite separable field
extensions of $K\llparenthesis t\rrparenthesis$ of the same degree.
Suppose that for some $l>\bd_{L}$, there is a $K\llbracket t\rrbracket$-isomorphism
$\cO_{L}/t^{l}\cO_{L}\xrightarrow{\sim}\cO_{L'}/t^{l}\cO_{L'}$. Then
there exists a $K\llparenthesis t\rrparenthesis$-isomorphism $L\xrightarrow{\sim}L'$. 
\end{cor}

\begin{proof}
From the last proposition, there exists a $K\llparenthesis t\rrparenthesis$-embedding
$L\to L'$. Because of their degrees, it is an isomorphism. 
\end{proof}
\begin{cor}
\label{cor:same-P}For $l\ge m$, the two morphisms $\psi^{(m)},\,\psi_{l}^{(m)}\circ\pi_{l}\colon\Eis^{(m)}\to\Delta_{n}^{\circ}$
induce the same P-morphism.
\end{cor}

\begin{proof}
For a geometric point $y\in\Eis^{(m)}(K)$, the field extension $A_{y}/K\llparenthesis t\rrparenthesis$
has discriminant exponent $m$. Since 
\[
\cO_{A_{y}}/t^{l+1}\cO_{A_{y}}=(K\llbracket t\rrbracket/(t^{l+1}))[x]/(f_{y}(x)),
\]
if $y,y'\in\Eis^{(m)}(K)$ map to the same point of $\Eis_{l}^{(m)}(K)$
for $l\ge m$, then $A_{y}$ and $A_{y'}$ are isomorphic over $K\llparenthesis t\rrparenthesis$.
This proves the corollary.
\end{proof}
\begin{cor}
\label{cor:compact}The locally closed subset $\Delta_{n}^{(m)}\subset\Delta_{n}^{\circ}$
is quasi-compact and constructible.
\end{cor}

\begin{proof}
We first note that from Remark \ref{rem:const-subsets}, the assertion
is independent of the choice of the P-moduli space $\Delta_{n}^{\circ}$
of the functor $\cF_{n}^{\circ}$. For an algebraically closed field
$K$, every totally ramified extension $A/K\tpars$ is associated
to some Eisenstein polynomial. This shows that 
\[
\Image(\psi^{(m)})=\Image(\psi_{l}^{(m)})=\Delta_{n}^{(m)}.
\]
Since the P-morphism $\psi_{l}^{(m)}$ is represented by a morphism
$Y\to\Delta_{n}^{\circ}$ of $k$-schemes with $Y$ a $k$-variety,
its image is quasi-compact and constructible.
\end{proof}
\begin{defn}
Let $\GG_{m}$ denote the multiplicative group $\Spec k[s,s^{-1}]$
over the base field $k$. We define a grading on $k[x_{1},\dots,x_{n}]$
by $\deg(x_{i})=i$ which induces a $\GG_{m}$-action on $V$ and
one on $\J_{\infty}V$. 
\end{defn}

\begin{lem}
Maps $\psi^{(m)}$ and $\psi_{l}^{(m)}$ are $\GG_{m}$-invariant. 
\end{lem}

\begin{proof}
Let $y=(y_{1},\dots,y_{n})\in\Eis^{(m)}(K)\subset K\llbracket t\rrbracket^{n}$
and $b\in K^{*}=\GG_{m}(K)$. Let 
\[
y':=by=(by_{1},b^{2}y_{2},\dots,b^{n}y_{n}).
\]
We have an isomorphism $A_{y}\to A_{y'},\,x\mapsto b^{-1}x$. This
shows that $\psi^{(m)}$ is $\GG_{m}$-invariant. Similarly for $\psi_{l}^{(m)}$.
\end{proof}
\begin{prop}
\label{prop:closed}Let $A\colon\Spec K\to\Delta_{n}^{(m)}$ be a
point with $K$ any field and let $l\ge m$. The fiber $\psi_{l}^{-1}(A)$
is a closed subset of $\Eis_{l}^{(m)}\otimes_{k}K$. 
\end{prop}

\begin{proof}
Since the proof is a little long and technical, we divide it into
several steps.

\textbf{A setup.} By base change, we may assume that $K=k$ and that
they are algebraically closed. Then, the point $A\in\Delta_{n}^{(m)}(k)$
is identified with an étale $k\tpars$-algebra of degree $n$. Let
$Z=\Spec R$ be an affine smooth irreducible curve over $k$ with
a distinguished closed point $0\in\Spec R$ and let $y\colon Z\to\Eis_{l}^{(m)}$
be a $k$-morphism such that an open dense subset $Z^{\circ}\subset Z$
map into $\psi_{l}^{-1}(A)$. We will show that the distinguished
point $0$ also maps into $\psi_{l}^{-1}(A)$, which implies the proposition.
Let $A_{y}=R\tpars[x]/(f_{y}(x))$ be the algebra corresponding to
$y$. Each point $z\in Z(k)$ induces a homomorphism $R\tpars\to k\tpars$,
which in turn induces an algebra $A_{y,z}/k\tpars$. For $z\in Z^{\circ}(k)$,
we have $A_{y,z}\cong A$. Our goal is to show that $A_{y,0}\cong A$.

\textbf{Strategy.} Our strategy to achieve this goal is as follows.
For a group $G$ of a special form, we construct some $G$-torsor
$\Spec B_{y}\to\Spec R\tpars$ which factors through $\Spec A_{y}$.
We show that for each point $z\in Z(k)$, the ``fiber'' $B_{y,z}$
is isomorphic to $(\widetilde{A_{y,z}})^{v}$ for some $v\in\ZZ_{>0}$,
where $\widetilde{A_{y,z}}$ denotes a Galois closure of $A_{y,z}/k\tpars$.
For such a group $G$, we have a fine moduli stack $\widetilde{\Delta}_{G}$
of $G$-torsors over $\Spec k\tpars$ and have the morphism $Z\to\widetilde{\Delta}_{G}$
corresponding to the above $G$-torsor. We see that the image of the
map $Z(k)\to\widetilde{\Delta}_{G}(k)/{\cong}$ is finite and hence
the map is constant. This implies the desired conclusion. 

\textbf{Construction of $B_{y}$.} We begin to construct a $G$-torsor
as above. Let $C_{y}/R\tpars$ and $C_{y,0}/k\tpars$ be the $S_{n}$-torsors
corresponding to the degree-$n$ étale algebras $A_{y}/R\tpars$ and
$A_{y,0}/k\tpars$, respectively. If we identify $S_{n-1}$ with the
stabilizer of $1$ for the action $S_{n}\curvearrowright\{1,\dots,n\}$,
then $(C_{y})^{S_{n-1}}=A_{y}$ and $(C_{y,0})^{S_{n-1}}=A_{y,0}$.
Let $\Spec B_{y,0}'\subset\Spec C_{y,0}$ be a connected component
and let $G\subset S_{n}$ be its stabilizer so that $\Spec B_{y,0}'\to\Spec k\tpars$
is a $G$-torsor and $B_{y,0}'/k\tpars$ is a Galois closure of $A_{y,0}/k\tpars$.
Since the residue field $k$ of $k\tpars$ is algebraically closed,
$G$ coincides with its inertia group (often denoted by $G_{0}$).
From \cite[p. 68, Corollary 4]{serre1979localfields}, $G=G_{0}$
is the semidirect product $H\rtimes C$ of a $p$-group $H$ and a
cyclic group $C$ of order coprime to the characteristic of $k$.
If $P$ denotes the stabilizer of 1 for the action $G\curvearrowright\{1,\dots,n\}$,
then $A_{y,0}=(B_{y,0}')^{P}$. 

Let $u\in A_{y}$ be the image of $x$ by the map
\[
R\tpars[x]\to A_{y}=R\tpars[x]/(f_{y}(x))
\]
and let $\overline{u}$ be the image of $u$ in $A_{y,0}$. Then,
$A_{y,0}=k\tpars[\overline{u}]$ and $B_{y,0}'=k\tpars[g\overline{u}\mid g\in G]$.
Let $B'_{y}:=R\tpars[gu\mid g\in G]\subset C_{y}$. Then $B_{y}'$
is a finitely generated torsion-free $R\tpars$-module. Note that
$R\tbrats$ is a Noetherian domain of dimension 2. Moreover, $R\tbrats$
is excellent \cite{valabrega1975onthe} and \cite[Theorem 19.5]{matsumura1987commutative}.
It follows that the localization $R\tpars=R\tbrats_{t}$ is an excellent
Dedekind domain. Since $B_{y}'$ is a torsion-free $R\tpars$-module
and $B_{y,0}'$ is a free $k\tpars$-module of rank $|G|$, $B_{y}'$
is a locally free $R\tpars$-module of rank $|G|$. Moreover, the
subring $B_{y}'\subset C_{y}$ is stable under the $G$-action. 

From \cite[Scholie 7.8.3]{grothendieck1965elements}, the ring $B_{y}'$
is excellent. Therefore, the normalization $B_{y}$ of $B_{y}'$ is
finitely generated as a $B_{y}'$-module as well as an $R\tpars$-module.
By a similar reasoning as above, $B_{y}$ is a locally free $R\tpars$-module
of rank $|G|$. We also see that $B_{y}$ is a locally free $A_{y}$-module
of rank $|G'|=|G|/n$. We have natural morphisms of Dedekind schemes
\begin{equation}
\Spec C_{y}\xrightarrow{\gamma}\Spec B_{y}\xrightarrow{\beta}\Spec A_{y}\xrightarrow{\alpha}\Spec R\tpars.\label{eq:covers}
\end{equation}
Morphisms $\alpha,\beta,\gamma$ as well as their compositions are
flat. The morphism $\alpha\circ\beta$ is $G$-invariant and the morphism
$\beta$ is $G'$-invariant. We know that $\alpha$ and $\alpha\circ\beta\circ\gamma$
are étale. From \cite[Proposition 17.7.7]{grothendieck1967elements},
$\alpha\circ\beta$ is étale. From \cite[Proposition 17.3.4]{grothendieck1967elements},
$\beta$ is also étale. 

\textbf{Proving that $\Spec B_{y}\to\Spec R\tpars$ is a $G$-torsor.}
We now claim that the $G$-action on $\Spec B_{y}$ is free. Each
element $g\in G\setminus\{1\}$ acts nontrivially on $B_{y,0}'$ and
hence also on $B_{y}'$ and $B_{y}$. If the $G$-action on $\Spec B_{y}$
is not free, then for some $g\in G\setminus\{1\}$, the quotient morphism
$\Spec B_{y}\to(\Spec B_{y})/\langle g\rangle$ is ramified. But,
this is impossible from \cite[Proposition 17.3.3(v)]{grothendieck1967elements}
and the fact that the morphism $\alpha\circ\beta\colon\Spec B_{y}\to\Spec R$
is étale and factors through $(\Spec B_{y})/\langle g\rangle$. This
shows the claim. We have showed that $\Spec B_{y}\to\Spec R\tpars$
is an étale finite morphism of degree $|G|$ and $G$-invariant for
a free $G$-action on $\Spec B_{y}$. We conclude that this is a $G$-torsor.
By the same argument, we also see that $\Spec B_{y}\to\Spec A_{y}$
is a $G'$-torsor. In summary, we have showed that some compositions
of morphisms in (\ref{eq:covers}) are torsors for groups displayed
in the following diagram:

\medskip{}
\[
\xymatrix{\Spec C_{y}\ar[r]_{\gamma}\ar@/^{2pc}/[rrr]^{S_{n}} & \Spec B_{y}\ar[r]_{\beta}^{G'}\ar@/_{2pc}/[rr]^{G} & \Spec A_{y}\ar[r]_{\alpha} & \Spec R\tpars}
\]
\medskip{}

\textbf{``Fibers'' of $B_{y}$ are powers of Galois closures.} For
a point $z\in Z(k)$, base changing (\ref{eq:covers}) with $\Spec k\tpars\to\Spec R\tpars$,
we get a sequence of finite étale morphisms
\[
\Spec C_{y,z}\xrightarrow{\gamma_{z}}\Spec B_{y,z}\xrightarrow{\beta_{z}}\Spec A_{y,z}\xrightarrow{\alpha_{z}}\Spec k\tpars.
\]
If $\widetilde{A_{y,z}}/k\tpars$ denotes a Galois closure of $A_{y,z}/k\tpars$,
then $C_{y,z}\cong(\widetilde{A_{y,z}})^{u}$ for some $u\in\ZZ_{>0}$.
Since $\Spec B_{y,z}\to\Spec k\tpars$ is a $G$-torsor, we can write
$B_{y,z}\cong D^{v}$ for some Galois extension $D/k\tpars$. Then,
$D$ is an intermediate field of $\widetilde{A_{y,z}}/A_{y,z}$ such
that $D/k\tpars$ is Galois, which shows that $D=\widetilde{A_{y,z}}$
and $B_{y,z}\cong(\widetilde{A_{y,z}})^{v}$.

\textbf{Completing the proof by using a moduli stack.} Since $G$
is the semidirect product $H\rtimes C$ of a $p$-group $H$ and a
tame cyclic group $C$, we can use the moduli stack $\widetilde{\Delta}_{G}$
of $G$-torsors over $\Spec k\tpars$ constructed in \cite{tonini2020moduliof}
(denoted by $\Delta_{G}$ in the cited paper). The stack $\widetilde{\Delta}_{G}$
is written as the inductive limit of Deligne--Mumford stacks $\cX_{n}$,
$n\in\ZZ_{\ge0}$ of finite type, where transition morphisms $\cX_{n}\to\cX_{n+1}$
are representable and universally injective. In particular, we have
\begin{align*}
\{G\text{-torsors over \ensuremath{\Spec k\tpars}}\}/{\cong} & =\Delta_{G}(k)/{\cong}\\
 & =\bigcup_{n}\cX_{n}(k)/{\cong}.
\end{align*}
Since $Z$ is quasi-compact, the morphism $Z\to\widetilde{\Delta}_{G}$
corresponding to the $G$-torsor $\Spec B_{y}\to\Spec R\tpars$ factors
through a morphism $Z\to\cX_{n}$ for some $n$. Recall that for $z\in Z^{\circ}(k)$,
we have an $k\tpars$-isomorphism $A_{y,z}\cong A$. Therefore, for
$z\in Z^{\circ}(k)$, $\widetilde{A_{y,z}}\cong\widetilde{A},$ where
$\widetilde{A}$ is a Galois closure of $A/k\tpars$. It follows that
for $z\in Z^{\circ}(k)$, $B_{y,z}\cong(\widetilde{A})^{v}$ as a
$k\tpars$-algebra. There are at most finitely many $G$-torsor structures
which can be given to the morphism $\Spec(\widetilde{A})^{v}\to\Spec k\tpars$.
It follows that the image of the map $Z^{\circ}(k)\to\cX_{n}(k)/{\cong}$
is a finite set. Since $(Z\setminus Z^{\circ})(k)$ is a finite set,
the image of the map $Z(k)\to\cX_{n}(k)/{\cong}$ is also finite.
Since $Z$ is irreducible, we conclude that the map $Z(k)\to\cX_{n}(k)/{\cong}$
is constant. Hence, the $G$-torsors $B_{y,z}/k\tpars$, $z\in Z(k)$
are isomorphic to one another. It follows that étale $k\tpars$-algebras
$A_{y,z}=(B_{y,z})^{G'}$, $z\in Z(k)$ are isomorphic to one another,
and hence all of them are isomorphic to $A$. In particular, $A_{y,0}\cong A$
as desired.
\end{proof}

\section{The space of uniformizers}

Let $A/K\tpars$ be a totally ramified extension of degree $n$ and
discriminant exponent $m$ with a chosen uniformizer $\varpi$. We
construct a map associating an Eisenstein polynomial to an arbitrary
uniformizer of $A$. To do so, we introduce some more notation. Let
$\widetilde{A}$ be a Galois closure of $A/K\llparenthesis t\rrparenthesis$
with Galois group $G$. We define groups
\begin{align*}
E & :=\{g\in G|\forall a\in A,\,g(a)=a\},\\
\widetilde{H} & :=\{g\in G\mid g(A)=A\},\\
H & :=\widetilde{H}/E=\Aut(A/K\tpars).
\end{align*}
Let $\sigma_{1},\dots,\sigma_{n}\colon A\hookrightarrow\tilde{A}$
be the $K\llparenthesis t\rrparenthesis$-embeddings. Let $s_{i}(u_{1},\dots,u_{n})$,
$i=1,\dots,n$, be elementary symmetric polynomials of $n$ variables
with degree $i$ and put $s_{0}=1$ by convention. 

With the above notation, to a uniformizer $\varpi'\in A$, we associate
the polynomial 
\[
f_{\varpi'}(x):=\sum_{i=0}^{n}(-1)^{i}s_{i}(\sigma_{1}(\varpi'),\dots,\sigma_{n}(\varpi'))x^{n-i}\in K\llbracket t\rrbracket[x].
\]
Note that the coefficients $(-1)^{i}s_{i}(\sigma_{1}(\varpi'),\dots,\sigma_{n}(\varpi'))$
are elements of $K\tbrats$, since they are invariant under the $G$-action.
If we denote the unique extension of the valuation $\ord$ to $A$
again by $\ord$, then
\begin{align*}
\ord s_{n}(\sigma_{1}(\varpi'),\dots,\sigma_{n}(\varpi')) & =\ord\prod_{i=1}^{n}\sigma_{i}(\varpi')=\sum_{i=1}^{n}\ord\sigma_{i}(\varpi')=n\cdot\frac{1}{n}=1.
\end{align*}
We also see that $s_{i}(\sigma_{1}(\varpi'),\dots,\sigma_{n}(\varpi'))$
has positive order for every $i>0$. There exists a $K\tpars$-isomorphism
\[
K\llparenthesis t\rrparenthesis[x]/(f_{\varpi'})\to A,\,x\mapsto\varpi'.
\]
Thus, we have obtained the map 
\[
\varpi'\mapsto f_{\varpi'}(x)
\]
sending a uniformizer to an Eisenstein polynomial. Conversely, if
an Eisenstein polynomial $f(x)\in K\tbrats[x]$ defines an extension
$K\tpars[x]/(f(x))$ admitting a $K\tpars$-isomorphism $\rho\colon K\tpars[x]/(f(x))\xrightarrow{\sim}A$,
then $f(x)$ is recovered as $f_{\rho(x)}$. Namely, the map
\begin{equation}
\{\text{uniformizers of \ensuremath{A}}\}\to\{y\in\Eis^{(m)}(K)\mid A_{y}\cong A\},\,\varpi'\mapsto f_{\varpi'}\label{eq:surj}
\end{equation}
is surjective.

Next we realize this map as a map of arc spaces as follows. Let 
\[
W:=\Spec K\llbracket t\rrbracket[w_{0},\dots,w_{n-1}]=\AA_{K\tbrats}^{n}
\]
and let $\fO_{A}:=\J_{\infty}W$. For an extension $L/K$, we can
identify $\fO_{A}(L)$ with the integer ring $\cO_{A_{L}}$ of $A_{L}:=A\otimes_{K\tpars}L\tpars$
by the bijection 
\begin{align*}
(\J_{\infty}W)(L) & \to\cO_{A_{L}},\\
\gamma=(\gamma_{0},\dots,\gamma_{n-1}) & \mapsto\varpi_{\gamma}:=\sum_{a=0}^{n-1}\gamma_{a}\varpi^{a}.
\end{align*}
Let $\Ut_{A},\Uf_{A},\fM_{A}\subset\fO_{A}$ be the locally closed
subschemes corresponding to the groups of units, the sets of uniformizers
and the maximal ideals of $\cO_{A_{L}}$ by the above bijection. For
each $l\in\ZZ_{\ge0}$, let $\fO_{A,l}:=\J_{l}W$, the $l$-jet scheme,
which corresponds to $\cO_{A}/t^{l+1}\cO_{A}$, and let $\pi_{l}\colon\fO_{A}\to\fO_{A,l}$
be the truncation map. We put 
\[
\Ut_{A,l}:=\pi_{l}(\Ut_{A})\text{ and }\Uf_{A,l}:=\pi_{l}(\Uf_{A}).
\]

Using the fixed uniformizer $\varpi\in A$, for each $1\le i\le n$,
we define the polynomial
\[
S_{i}(w_{0},\dots,w_{n-1}):=s_{i}\left(\sum_{a=0}^{n-1}w_{a}\sigma_{1}(\varpi^{a}),\dots,\sum_{a=0}^{n-1}w_{a}\sigma_{n}(\varpi^{a})\right)\in\cO_{\tilde{A}}[w_{0},\dots,w_{n-1}].
\]
This is invariant by the $G$-action, hence in fact belongs to $K\llbracket t\rrbracket[w_{0},\dots,w_{n-1}]$. 
\begin{defn}
We define a $K\llbracket t\rrbracket$-morphism 
\[
\phi\colon W=\Spec K\tbrats[w_{0},\dots,w_{n-1}]\to V_{K}=\Spec K\tbrats[x_{1},\dots,x_{n}]
\]
by $\phi^{*}(x_{i})=(-1)^{i}S_{i}$. We write the associated maps
of arc spaces and jet schemes as 
\[
\phi_{l}\colon\J_{l}W=\fO_{A,l}\to\J_{l}V_{K}\quad(0\le l\le\infty).
\]
\end{defn}

We define gradings on $K\llbracket t\rrbracket[w_{0},\dots,w_{n-1}]$
and $K\llbracket t\rrbracket[x_{1},\dots,x_{n}]$ by 
\[
\deg w_{i}=1\text{ and }\deg x_{i}=i,
\]
respectively. With these gradings, the map $\phi^{*}$ of coordinate
rings is degree-preserving and $\phi$ is equivariant for the corresponding
actions of $\GG_{m,K}=\GG_{m}\otimes_{k}K$. It follows that $\phi_{l}$,
$0\le l\le\infty$ are also equivariant for the induced actions of
$\GG_{m,K}$ on jet/arc spaces. 
\begin{lem}
\label{lem:image of Uni}Let $L/K$ be a field extension. The map
$\phi_{\infty}$ sends $\gamma\in\Uf_{A}(L)$ to the point of $\Eis^{(m)}(L)$
corresponding to $f_{\varpi_{\gamma}}$. In particular, we have $\phi_{\infty}(\Uf_{A})=\psi^{-1}(A)$
and $\phi_{l}(\Uf_{A,l})=\psi_{l}^{-1}(A)$ for $l\ge m$. 
\end{lem}

\begin{proof}
The first assertion follows from construction. For a field extension
$L/K$, let $A_{L}:=A\otimes_{K\tpars}L\tpars$. As a base change
of map (\ref{eq:surj}), we get a surjection
\[
\{\text{uniformizers of \ensuremath{A_{L}} }\}\to\{y\in\Eis^{(m)}(L)\mid A_{y}\cong A_{L}\},\,\varpi'\mapsto f_{\varpi'}.
\]
This shows $\phi_{\infty}(\Uf_{A})=\psi^{-1}(A)$. We get
\[
\phi_{l}(\Uf_{A,l})=\phi_{l}(\pi_{l}(\Uf_{A}))=\pi_{l}(\phi_{\infty}(\Uf_{A}))=\pi_{l}(\psi^{-1}(A))=\psi_{l}^{-1}(A),
\]
where the last equality follows from Corollary \ref{cor:same-P}.
\end{proof}
Let $\sigma\in H=\Aut(A/K\tpars)$ and for each $0\le a<n$, let us
write 
\[
\sigma(\varpi^{a})=\sum_{b=0}^{n-1}c_{ab}^{\sigma}\varpi^{b}\quad(c_{ab}^{\sigma}\in K\tbrats).
\]
Let $\sigma$ act on $K\llbracket t\rrbracket[w_{0},\dots,w_{n-1}]$
by $w_{b}\mapsto\sum_{a=0}^{n-1}c_{ab}^{\sigma}w_{a}$. This gives
an $H$-action on $W$.
\begin{lem}
\label{lem:actions commute}This $H$-action on $W$ commutes with
the $\GG_{m,K}$-action and the morphism $\phi$ is $H$-invariant. 
\end{lem}

\begin{proof}
The first assertion follows from the fact that the automorphism of
$K\llbracket t\rrbracket[w_{0},\dots,w_{n-1}]$ induced by $\sigma\in H$
is linear and degree-preserving. To show the second assertion, it
is enough to show that $S_{i}$ are $H$-invariant. For $\sigma,\sigma'\in H$,
we have
\begin{align*}
\sum_{b=0}^{n-1}\sigma(w_{b})\sigma'(\varpi^{b}) & =\sum_{b=0}^{n-1}\sum_{a=0}^{n-1}c_{ab}^{\sigma}w_{a}\sigma'(\varpi^{b})\\
 & =\sum_{a=0}^{n-1}w_{a}\sum_{b=0}^{n-1}c_{ab}^{\sigma}\sigma'(\varpi^{b})\\
 & =\sum_{a=0}^{n-1}w_{a}(\sigma'\sigma)(\varpi^{a}).
\end{align*}
Since $s_{i}$ are symmetric polynomials and sequences $(\sigma\sigma_{1})(\varpi^{a}),\dots,(\sigma\sigma_{m})(\varpi^{a})$
and $\sigma_{1}(\varpi^{a}),\dots,\sigma_{m}(\varpi^{a})$ are interchanged
by permutations, we have
\begin{align*}
S_{i}(\sigma(w_{0}),\dots,\sigma(w_{n-1})) & =s_{i}\left(\sum_{a=0}^{n-1}\sigma(w_{a})\sigma_{1}(\varpi^{a}),\dots,\sum_{a=0}^{n-1}\sigma(w_{a})\sigma_{n}(\varpi^{a})\right)\\
 & =s_{i}\left(\sum_{a=0}^{n-1}w_{a}(\sigma\sigma_{1})(\varpi^{a}),\dots,\sum_{a=0}^{n-1}w_{a}(\sigma\sigma_{n})(\varpi^{a})\right)\\
 & =S_{i}(w_{0},\dots,w_{n-1}).
\end{align*}
Thus $\phi$ is $H$-invariant. 
\end{proof}
The last lemma implies that the morphism $\phi_{\infty}\colon\J_{\infty}W\to\J_{\infty}V_{K}$
is also $H$-invariant and so is its restriction $\Uf_{A}\to\psi^{-1}(A)$.
Note that from Proposition \ref{prop:closed}, $\psi_{l}^{-1}(A)\subset\J_{l}V_{K}$,
$l\ge m$ and $\psi^{-1}(A)=\pi_{l}^{-1}(\psi_{l}^{-1}(A))\subset\J_{\infty}V_{K}$
are closed subsets. We obtain morphisms
\[
\overline{\phi_{\infty}}\colon\Uf_{A}/H\to\psi^{-1}(A)\text{ and }\overline{\phi_{l}}\colon\Uf_{A,l}/H\to\psi_{l}^{-1}(A)\quad(l\in\ZZ_{\ge0}).
\]

\begin{lem}
\label{lem:Uni/H to Eis}For each extension $L/K$, the map $\Uf_{A}(L)/H\to\psi^{-1}(A)(L)$
is bijective.
\end{lem}

\begin{proof}
For a uniformizer $\varpi'$, the associated Eisenstein polynomial
is also written as $f_{\varpi'}=\prod_{i=1}^{n}(x-\sigma_{i}(\varpi'))$.
Therefore, if $f_{\varpi'}=f_{\varpi''}$, then $\varpi'$ and $\varpi''$
are conjugate and $\varpi''=\sigma(\varpi')$ for some $\sigma\in H$.
Thus $\Uf_{A}(L)/H\to\psi^{-1}(A)(L)$ is injective. The surjectivity
follows from Lemma \ref{lem:image of Uni}. 
\end{proof}
We can summarize relations among various spaces that we obtained so
far in the following diagram. 

\[
\xymatrix{\Uf_{A}\ar@{^{(}->}[rr]\ar[d] &  & \J_{\infty}W\ar[dd]^{\phi_{\infty}}\\
\Uf_{A}/H\ar[d]_{\textrm{bij.}}^{\overline{\phi_{\infty}}}\\
\psi^{-1}(A)\ar@{^{(}->}[r]\ar[d] & \Eis^{(m)}\otimes_{k}K\ar@{^{(}->}[r]\ar[d]^{\psi\otimes_{k}K} & \J_{\infty}V\\
A\ar@{}[r]|{\in} & \Delta_{n}^{\circ}\otimes_{k}K
}
\]

Let $e$ be the ramification index of $\widetilde{A}/K\tpars$ and
let $\mathfrak{D}_{\widetilde{A}/K\tpars}$ and $\fD_{A/K\tpars}$
be the differents of $A/K\tpars$ and $\widetilde{A}/K\tpars$, which
are principal ideals of $\cO_{A}$ and $\cO_{\widetilde{A}}$, respectively.
We denote the normalized valuation on $\widetilde{A}$ by $v_{\widetilde{A}}$
and the unique extension of the valuation $\ord$ to $\widetilde{A}$
again by $\ord$. The two valuations are related by $e\cdot v_{\widetilde{A}}=\ord$. 
\begin{lem}
\label{lem:different of Gal closure}We have
\[
\frac{v_{\tilde{A}}(\fD_{\tilde{A}/K\llparenthesis t\rrparenthesis})}{e}=\ord\fD_{\tilde{A}/K\llparenthesis t\rrparenthesis}\le n\cdot\ord\fD_{A/K\llparenthesis t\rrparenthesis}=m.
\]
\end{lem}

\begin{proof}
The left equality is obvious. The right equality follows from the
fact that the discriminant is the norm of the different and a formula
for valuations of norms (for example, \cite[Chapter II, (4.8) and Chapter III, (2.9)]{neukirch1999algebraic}).
The Galois closure $\tilde{A}$ is obtained as the composite of all
conjugates of $A$ in an algebraic closure of $K\llparenthesis t\rrparenthesis$.
This shows $\fD_{\tilde{A}/K\llparenthesis t\rrparenthesis}\mid(\fD_{A/K\llparenthesis t\rrparenthesis})^{n}$.
Indeed Toyama \cite{t^oyama1955anote2} proved the corresponding result
for the composite of two number fields. The same argument applies
to the case of local fields and the generalization to the composite
of an arbitrary number of local fields is straightforward. It follows
that $\ord\fD_{\tilde{A}/K\llparenthesis t\rrparenthesis}\le n\cdot\ord\fD_{A/K\llparenthesis t\rrparenthesis}$. 
\end{proof}
\begin{lem}
\label{lem:free action}For $l\ge m$, the $H$-action on $\Uf_{A,l}$
is free.
\end{lem}

\begin{proof}
By taking a base change, we may assume that $K$ is algebraically
closed, in particular, perfect. Then, we can apply the theory of ramifications
groups explained in \cite[Ch.apter IV]{serre1979localfields}; we
follow the notation from this reference. From \cite[p.  64,  Proposition 4]{serre1979localfields},
for $1\ne g\in G$, $i_{G}(g)\le v_{\tilde{A}}(\fD_{\tilde{A}/K\llparenthesis t\rrparenthesis})$.
Let $a=em$ and let $G_{a}$ be the $a$-th ramification group (lower
numbering) consisting of $g$ with $i_{G}(g)\ge a+1$. From Lemma
\ref{lem:different of Gal closure}, for $g\ne1$, $i_{G}(g)\le a$
and $g\notin G_{a}$. Namely $G_{a}=1$. Let $\varphi_{\tilde{A}/K\llparenthesis t\rrparenthesis}$
be the Herbrand function so that $G_{a}=G^{\varphi_{\tilde{A}/K\llparenthesis t\rrparenthesis}(a)}$,
where the right hand side is a higher ramification group with the
upper numbering. From \cite[p.64, Proposition 4 and p. 74, Lemma 3]{serre1979localfields},
we have

\begin{align*}
\varphi_{\tilde{A}/K\llparenthesis t\rrparenthesis}(a) & =\frac{1}{e}\sum_{g\in G}\min\{i_{G}(g),a+1\}-1\\
 & =\frac{1}{e}\sum_{g\ne1}i_{G}(g)+\frac{a+1}{e}-1\\
 & =\frac{v_{\tilde{A}}(\fD_{\tilde{A}/K\llparenthesis t\rrparenthesis})}{e}+\frac{a+1}{e}-1\\
 & \le m+\frac{em+1}{e}-1\\
 & \le2m.
\end{align*}
Thus $G^{2m}=1$. From \cite[Proposition A.6.1]{deligne1984lescorps}
and \cite[p.61,  Lemma 1]{serre1979localfields}, if $\sigma\colon A\hookrightarrow\tilde{A}$
is a $K\tpars$-embedding different from the chosen embedding $A\subset\widetilde{A}$
and if $\varpi'\in\cO_{A}$ is a uniformizer, then 
\begin{align*}
2\cdot\ord(\sigma(\varpi')-\varpi') & \le\sum_{\tau\ne\sigma_{0}}\ord(\tau(\varpi')-\varpi')+\sup_{\tau\ne\sigma_{0}}\ord(\tau(\varpi')-\varpi')\\
 & <2m+1.
\end{align*}
 and hence $\ord(\sigma(\varpi')-\varpi')<m+1$. In particular, for
$1\ne h\in H$, if $l\ge m$, then
\[
h(\varpi')\ne\varpi'\mod t^{l+1}.
\]
This shows the lemma. 
\end{proof}

\section{Bundles having ``almost affine spaces'' as fibers}

We keep the notation from the last section. In particular, $A$ denotes
a totally ramified extension of $K\tpars$ of degree $n$ and discriminant
exponent $m$. Let $\cJ\subset K\llbracket t\rrbracket[w_{0},\dots,w_{n-1}]$
be the Jacobian ideal of the morphism $\phi\colon W\to V$. Namely,
$\cJ$ is the principal ideal generated by the determinant of the
Jacobian matrix 
\[
\frac{\partial(S_{1},\dots,S_{n})}{\partial(w_{0},\dots,w_{n-1})}.
\]
Its order function 
\[
\ord_{\cJ}\colon\J_{\infty}W\to\ZZ_{\ge0}\cup\{\infty\}
\]
is defined as follows: for a point $\gamma\in\J_{\infty}W$ represented
by a morphism $\Spec L\to\J_{\infty}W$ with $L$ a field, if $\gamma'\colon\Spec L\tbrats\to W$
is the corresponding arc and if we can write $(\gamma')^{-1}\cJ=(t^{n})\subset L\tbrats$,
then $\ord_{\cJ}(\gamma)=n$. Here, we followed the convention that
$(0)=(t^{\infty})$. 
\begin{lem}
The restriction $\ord_{\cJ}|_{\Uf_{A}}$ of $\ord_{\cJ}$ to $\Uf_{A}$
takes the constant value $m$. 
\end{lem}

\begin{proof}
Since $S_{i}$ is the composite of $s_{i}(x_{1},\dots,x_{n})$ and
$x_{j}=\sum_{a}w_{a}\sigma_{j}(\varpi^{a})$, the above Jacobian matrix
can be written as 
\begin{equation}
\left.\frac{\partial(s_{1},\dots,s_{n})}{\partial(x_{1},\dots,x_{n})}\right|_{x_{i}=\sum_{a=0}^{n-1}w_{a}\sigma_{i}(\varpi^{a})}\times\frac{\partial\left(\sum_{a=0}^{n-1}w_{a}\sigma_{1}(\varpi^{a}),\dots,\sum_{a=0}^{n-1}w_{a}\sigma_{n}(\varpi^{a})\right)}{\partial(w_{0},\dots,w_{n-1})}.\label{eq:factor}
\end{equation}
Concerning the first factor, as was proved in \cite[p.  1036]{serre1978unetextquotedblleftformule},
we have 
\[
\det\frac{\partial(s_{1},\dots,s_{n})}{\partial(x_{1},\dots,x_{n})}=\prod_{i<j}(x_{i}-x_{j}).
\]
For $\gamma=(\gamma_{0},\dots,\gamma_{n-1})\in\Uf_{A}$, substituting
$\gamma_{a}$ for $w_{a}$ in 
\[
\left.\left(\prod_{i<j}(x_{i}-x_{j})\right)\right|_{x_{i}=\sum_{a=0}^{n-1}w_{a}\sigma_{i}(\varpi^{a})},
\]
we get
\[
\prod_{i<j}(\sigma_{i}(\varpi_{\gamma})-\sigma_{j}(\varpi_{\gamma})).
\]
From \cite[the proof of Proposition 1.29]{keune2023numberfields},
its square is the discriminant of $1,\varpi_{\gamma},\dots,\varpi_{\gamma}^{n-1}$.

On the other hand, concerning the second factor of (\ref{eq:factor}),
we have
\[
\frac{\partial\left(\sum_{a=0}^{n-1}w_{a}\sigma_{1}(\varpi^{a}),\dots,\sum_{a=0}^{n-1}w_{a}\sigma_{n}(\varpi^{a})\right)}{\partial(w_{0},\dots,w_{n-1})}=(\sigma_{i}(\varpi^{a}))_{i,a}.
\]
The discriminant of $A/k\tpars$ with respect to the basis $1,\varpi,\dots,\varpi^{n-1}$
is the square of $\det(\sigma_{i}(\varpi^{a}))_{i,a}$ \cite[Proposition 1.28]{keune2023numberfields}.
We have showed that the determinants of the two factors in (\ref{eq:factor})
both have order $m/2$, which shows $\ord_{\cJ}(\gamma)=m$. 
\end{proof}
\begin{lem}
For an extension $L/K$ and for $l\ge2m$, every fiber of $\Uf_{A,l}(L)/H\to\Eis_{l}^{(m)}(L)$
is contained in a fiber of $\Uf_{A,l}(L)/H\to\Uf_{l-m}(L)/H$.
\end{lem}

\begin{proof}
Suppose that two points $b_{l},b_{l}'\in\Uf_{A,l}(L)$ map to the
same point $a_{l}\in\Eis_{l}^{(m)}(L)$. Let $b,b'\in\Uf_{A}(L)$
be lifts of $b_{l},b_{l}'$ respectively respectively and let $a:=\phi_{\infty}(b)\in\Eis^{(m)}(L)$.
Then, from \cite[Chapter 5,  Proposition 3.1.7]{chambert-loir2018motivic},
there exists $c\in\Uf_{A}$ such that $\phi_{\infty}(c)=a$ and the
images $c_{l-m},b'_{l-m}\in\Uf_{A,l-m}$ of $c,b'$ are the same.
Note that the constant $c_{V}$ appearing in the cited result is taken
to be $1$ in our situation, since $V$ is smooth. The condition that
$\phi_{\infty}(c)=\phi_{\infty}(b)=a$ implies that $c$ is in the
$H$-orbit of $b$. Therefore $H(c_{l-m})=H(b_{l-m})=H(b_{l-m}')$.
This shows the lemma.
\end{proof}
\begin{defn}
\label{def:bundle}Let $r$ be a non-negative integer. We say that
a morphism $f\colon Y\to X$ of $k$-varieties is an\emph{ $\AA^{r}$-bundle
}(resp\emph{.~pseudo-$\AA^{r}$-bundle}) if for every geometric point
$x\in X(K)$, the fiber $f^{-1}(x)$ is $K$-isomorphic to $\AA_{K}^{r}$
(resp.~to the quotient $\AA_{K}^{r}/C$ of the affine space $\AA_{K}^{n}$
by a finite cyclic group $C$ of order coprime to $p$).
\end{defn}

\begin{defn}
A variety $X$ over $k$ is said to be a \emph{weak affine space}
if there exists a sequence of $k$-varieties $X=X_{0},X_{1},\dots,X_{l}=\Spec k$
such that for each $0\le i<l$, there exists a morphism between $X_{i}$
and $X_{i+1}$ in either direction which is a weak $\AA^{r_{i}}$-bundle
for some $r_{i}$. A weak affine space of dimension $r$ is called
also as a \emph{weak $\AA_{k}^{r}$}. 
\end{defn}

\begin{lem}
\label{lem:A^m fib}For $l\ge2m$, the morphism $\overline{\phi_{l}}\colon\Uf_{A,l}/H\to\psi_{l}^{-1}(A)$
is an $\AA^{m}$-bundle. 
\end{lem}

\begin{proof}
By base change, we may suppose that $K=k$ and they are algebraically
closed. Moreover, it suffices to consider fibers over closed points.
In the rest of this proof, all points that we consider are closed.
For $a_{l}\in\psi_{l}^{-1}(A)$, the fiber $(\overline{\phi_{l}})^{-1}(a_{l})$
of $\overline{\phi_{l}}\colon\Uf_{A,l}/H\to\Eis_{l}^{(m)}$ over $a_{l}$
is contained in the fiber $(\pi_{l-m}^{l})^{-1}(\overline{b_{l-m}})$
of the truncation map $\pi_{l-m}^{l}\colon\Uf_{A,l}/H\to\Uf_{A,l-m}/H$
over a point $\overline{b_{l-m}}\in\Uf_{A,l-m}/H$. Let $b_{l-m}\in\Uf_{A,l-m}$
be a lift of $\overline{b_{l-m}}$. From Lemma \ref{lem:free action},
since $l-m\ge m$, the $H$-action on $\Uf_{A,l-m}$ is free. In particular,
for $h\in H\setminus\{1\}$, $h(b_{l-m})\ne b_{l-m}$, which shows
that for two distinct points of $\pi_{l-m}^{-1}(b_{l-m})$, the $H$-action
cannot send one point to the other. From Lemma \ref{lem:Uni/H to Eis},
the restriction of $\phi_{\infty}$ to $\pi_{l-m}^{-1}(b_{l-m})$
is injective. The freeness of the $H$-action on $\Uf_{A,l-m}$ also
shows that the map $(\pi_{l-m}^{l})^{-1}(b_{l-m})\to(\pi_{l-m}^{l})^{-1}(\overline{b_{l-m}})$
is an isomorphism. From \cite[Chapter 5, Theorem 3.2.2.c]{chambert-loir2018motivic},
$(\pi_{l-m}^{l})^{-1}(b_{l-m})\to\Eis_{l}^{(m)}$ is a piecewise $\AA^{m}$-bundle
over its image and so is $(\pi_{l-m}^{l})^{-1}(\overline{b_{l-m}})\to\Eis_{l}^{(m)}$.
Since $(\overline{\phi_{l}})^{-1}(a_{l})\subset(\pi_{l-m}^{l})^{-1}(\overline{b_{l-m}})$,
we conclude that $(\overline{\phi_{l}})^{-1}(a_{l})\cong\AA_{k}^{m}$. 
\end{proof}
\begin{lem}
\label{lem:tame-cyclic}Let $y$ be a geometric point of $\psi_{l}^{-1}(A)$.
Suppose $p\mid n$ and $l\ge\lfloor m/n\rfloor$. Then, the stabilizer
$\Stab(y)$ of $y$ with respect to the $\GG_{m}$-action on $\Eis_{l}^{(m)}$
is a tame cyclic group. 
\end{lem}

\begin{proof}
Let $L/k$ be the algebraically closed field such that $y$ is an
$L$-point. For $\lambda\in\GG_{m}(L)$, if we write $y=(y_{1},\dots,y_{n})\in(L\tbrats/(t^{l+1}))^{n}$,
we have
\[
\lambda\cdot\gamma=(\lambda\gamma_{1},\dots,\lambda^{n}\gamma_{n}).
\]
Thus, if $y_{i}\ne0$ for some $i$, then $\Stab(y)$ is contained
$\mu_{i}$, the group of $i$-th roots of $1$. From Proposition \ref{prop:Eis-m-explicit},
there exists $i\in\{1,\dots,n-1\}$ such that $p\nmid i$ and $y_{i}\ne0$,
which shows the first assertion of the lemma. 
\end{proof}
From Lemma \ref{lem:actions commute}, we get a $\GG_{m,K}$-equivariant
morphism $\Uf_{A,l}/H\to\psi_{l}^{-1}(A)$. Since $\Uf_{A,l}$ is
affine, so is $\Uf_{A,l}/H$. The fiber $\psi_{l}^{-1}(A)$ is also
an affine variety over $K$, thanks to Proposition \ref{prop:closed}.
The $\GG_{m,K}$-actions on $\psi_{l}^{-1}(A)$ and $\Uf_{A,l}/H$
have finite stabilizers. Thus, the geometric quotients $\Uf_{A,l}/(H\times\GG_{m,K})=(\Uf_{A,l}/H)/\GG_{m,K}$
and $\psi_{l}^{-1}(A)/\GG_{m,K}$ exist and are again affine varieties
over $K$. 
\begin{cor}
\label{cor:A^m fib mod H}For $l\ge2m$, the morphism $\Uf_{A,l}/(H\times\GG_{m,K})\to\psi_{l}^{-1}(A)/\GG_{m,K}$
is a pseudo-$\AA^{m}$-bundle.
\end{cor}

\begin{proof}
Let $y\in\psi_{l}^{-1}(A)(L)$ be a geometric point and let $\overline{y}$
be its image in $(\psi_{l}^{-1}(A)/\GG_{m,K})(L)$. Let $C$ be the
stabilizer at $y$ for the $\GG_{m,K}$-action on $\psi_{l}^{-1}(A)$,
which is a finite and tame cyclic group from Lemma \ref{lem:tame-cyclic}.
From \cite[Lemma 2.3.3]{abramovich2002compactifying}, the fiber of
$\Uf_{A,l}/(H\times\GG_{m,K})\to\psi_{l}^{-1}(A)/\GG_{m,K}$ over
$\overline{y}$ is the quotient of the fiber of $\Uf_{A,l}/H\to\psi_{l}^{-1}(A)$
over $y$ by the action of $C$. The corollary follows from Lemma
\ref{lem:A^m fib}.
\end{proof}
\begin{lem}
\label{lem:Unif/H*Gm}The structure morphism $\Uf_{A,l}/(H\times\GG_{m,K})\to\Spec K$
factors as 
\[
\Uf_{A,l}/(H\times\GG_{m})=Z_{n(l+1)-2}\to Z_{n(l+1)-3}\to\cdots\to Z_{0}=\Spec K
\]
in such a way that each morphism $Z_{i+1}\to Z_{i}$ is an $\AA^{1}$-bundle.
In particular, $\Uf_{A,l}/(H\times\GG_{m,K})$ is a weak $\AA_{K}^{n(l+1)-2}$. 
\end{lem}

\begin{proof}
Again by a base change argument, we may assume that $K=k$ and they
are algebraically closed and consider only closed points. Let $\fm_{A}$
be the maximal ideal of $A$ and let $\fC_{n(l+1)}\subset\Uf_{A,l}$
be the subvariety corresponding to elements of 
\[
\cO_{A}/t^{l+1}\cO_{A}=\cO_{A}/\fm_{A}^{n(l+1)}
\]
that are equal to the chosen uniformizer $\varpi$ modulo $\fm_{A}^{2}$.
For $2\le i\le n(l+1)$, let $\fC_{i}$ be the variety corresponding
to the images of such elements in $\cO_{A}/\fm_{A}^{i}$. As a $K$-variety,
$\fC_{i}$ is isomorphic to $\AA_{K}^{i-2}$. We define a (lower-numbering)
filtration $H_{*}$ of $H$ in the way which is standard if $A/K\tpars$
is a Galois extension: for $i\ge0$, 
\[
H_{i}:=\Ker(H\to\Aut(\cO_{A}/\fm_{A}^{i+1})).
\]
For a uniformizer $\varpi\in A$, the map $g\mapsto g(\varpi)/\varpi$
defines an injective homomorphism $H_{i}/H_{i+1}\to U^{i}/U^{i+1}$.
From \cite[Section 1.6]{serre1961surles}, $U^{0}/U^{1}\cong\GG_{m}$
and for $i\ge1$, $U^{i}/U^{i+1}\cong\GG_{a}$. Therefore $H_{0}/H_{1}$
is a tame cyclic group and $H_{i}/H_{i+1}$, $i\ge1$ are elementary
abelian $p$-groups. Similarly, for $i\ge1$, we can have an embedding
$H_{i}/H_{i+1}\to\fm_{L}^{i+1}/\fm_{L}^{i+2}$ by $g\mapsto g(\varpi)-\varpi$. 

We claim that the composite morphism 
\[
\fC_{n(l+1)}\hookrightarrow\Uf_{A,l}\to\Uf_{A,l}/\GG_{m}
\]
is an isomorphism. Indeed, the morphism is clearly a universal homeomorphism.
To show that it is an isomorphism, it is enough to show that it is
étale. We have the $\GG_{m}$-equivariant morphism $\Uf_{A,l}\to\GG_{m}$
corresponding to the map sending a uniformizer $\varpi'$ to $\lambda\in k^{*}$
if $\varpi'=\lambda\varpi$ modulo $\fm_{A}^{2}$. This is a smooth
morphism and $\fC_{n(l+1)}$ is the fiber over $1\in\GG_{m}$ of this
morphism. This shows that $\fC_{n(l+1)}$ intersects with each $\GG_{m}$-orbit
transversally, which implies that $\fC_{n(l+1)}\to\Uf_{A,l}/\GG_{m,k}$
is étale. We have proved the claim.

We see that $H(\fC_{n(l+1)})\subset\Uf_{A,l}$ consists of $|H/H_{1}|$
connected components, each of which is isomorphic to $\fC_{n(l+1)}$.
The stabilizer of the component $\fC_{n(l+1)}\subset H(\fC_{n(l+1)})$
is $H_{1}$. Therefore $\Uf_{A,l}/(H\times\GG_{m})\cong\fC_{n(l+1)}/H_{1}$.
We claim that for every $j\ge1$ and every $i$ with $2\le i\le n(l+1)$,
the map $\fC_{i}/H_{j}\to\fC_{i-1}/H_{j}$ is an $\AA^{1}$-bundle.
To see this, we first note that $\fC_{i}\to\fC_{i-1}$ is a bundle
with fiber $\fm_{L}^{i-1}/\fm_{L}^{i}=\GG_{a}$. If $j\ge i-1$, then
the $H_{j}$-actions on $\fC_{i}$ and $\fC_{i-1}$ are trivial and
the claim follows. If $j=i-2$, then the map of the claim has fibers
isomorphic to
\[
(\fm_{L}^{i-1}/\fm_{L}^{i})/(H_{i-2}/H_{i-1})\cong\GG_{a}.
\]
The last isomorphism follows from the fact that \cite[Corollary 14.56]{milne2017algebraic}.
For $j<i-2$, the map of the claim is identified with 
\[
(\fC_{i}/H_{i-2})/(H_{j}/H_{i-2})\to\fC_{i-1}/(H_{j}/H_{i-2}).
\]
The claim now holds since the action of $H_{j}/H_{i-2}$ on $\fC_{i-1}$
is free. We have proved the claim.

Thus we obtain a tower of $\AA^{1}$-bundles
\[
\Uf_{A,l}/(H\times\GG_{m})\cong\fC_{n(l+1)}/H_{1}\to\fC_{n(l+1)-1}/H_{1}\to\cdots\to\fC_{2}/H_{1}=\pt.
\]
This shows the lemma.
\end{proof}
\begin{lem}
\label{lem:closed-irreducible}For every geometric point $A\in\Delta_{n}^{\circ}(K)$
with $\bd_{A}=m$ and for $l\ge2m$, the closed subset $\psi_{l}^{-1}(A)/\GG_{m,K}\subset\Eis_{l}^{(m)}\otimes_{k}K/\GG_{m,K}$
is a weak $\AA_{K}^{n(l+1)-2-m}$. 
\end{lem}

\begin{proof}
We first note that from Proposition \ref{prop:closed}, $\psi_{l}^{-1}(A)\subset\Eis_{l}^{(m)}\otimes_{k}K$
is a closed subset, which is also $\GG_{m,K}$-invariant. Therefore,
$\psi_{l}^{-1}(A)/\GG_{m,K}\subset(\Eis_{l}^{(m)}\otimes_{k}K)/\GG_{m,K}$
is also a closed subset (for example, see \cite[Theorem 1.1]{newstead2009geometric2}).
From Corollary \ref{cor:A^m fib mod H} and Lemma \ref{lem:Unif/H*Gm},
$\psi_{l}^{-1}(A)/\GG_{m,K}$ is a weak $\AA_{K}^{n(l+1)-2-m}$. 
\end{proof}
\begin{lem}
\label{lem:technical}Let $Y,X$ be $k$-varieties and let $f\colon Y\to X$
be a P-morphism over $k$ with the graph $\Gamma\subset Y\times_{k}X$.
Let $\pr_{Y}$ and $\pr_{X}$ be the projections from $Y\times_{k}X$
to $Y$ and $X$, respectively. Suppose that for every point $w\colon\Spec K\to X$
with $K$ an arbitrary field, the fiber $f^{-1}(w)$ is an irreducible
closed subset of $Y\otimes_{k}K$. 
\begin{enumerate}
\item Then, for every point $x\in X$, there exists an open dense subset
$U_{x}$ of $\overline{\{x\}}$ such that $\pr_{X}^{-1}(U_{x})\cap\Gamma$
is a closed subset of $\pr_{X}^{-1}(U_{x})$. 
\item There exists a decomposition $X=\bigsqcup_{i=1}^{l}X_{i}$ of $X$
into finitely many locally closed subsets $X_{i}$ such that for each
$i$, $\pr_{X}^{-1}(X_{i})\cap\Gamma$ is a locally closed subset
of $Y\times_{k}X$.
\end{enumerate}
\end{lem}

\begin{proof}
(1) By assumption, $\Gamma\cap\pr_{X}^{-1}(x)$ is a closed subset
of $\pr_{X}^{-1}(x)$. Let $C$ be its closure in $Y\times X$, which
is contained in $\pr_{X}^{-1}(\overline{\{x\}})$. Since $C\cap\pr_{X}^{-1}(x)=\Gamma\cap\pr_{X}^{-1}(x)$,
the symmetric difference $C\triangle\Gamma$ is a constructible subset
whose image in $X$ does not contain $x$. Let $Z:=\overline{\{x\}}\cap\pr_{X}(C\triangle\Gamma)$.
This is a constructible subset of $\overline{\{x\}}$ which does not
contain $x$. Let $U_{x}:=\overline{\{x\}}\setminus\overline{Z}$,
which is an open dense subset of $\overline{\{x\}}$. By construction,
we have $C\cap\pr_{X}^{-1}(U_{x})=\Gamma\cap\pr_{X}^{-1}(U_{x})$,
which is a closed subset of $\pr_{X}^{-1}(U_{x})$. 

(2) We show this in the stronger form such that for every $j$, $W_{j}:=\bigsqcup_{i=j}^{l}X_{i}$
is a closed subset of $X$. We construct $X_{i}$'s inductively as
follows. We put $X_{1}$ to be $U_{\eta}$ for the generic point $\eta$
of an irreducible component of $W_{1}=B$ having the maximal dimension.
Suppose that we have constructed $X_{1},\dots,X_{j-1}$ and let $W_{j}=X\setminus\bigcup_{i=1}^{j-1}X_{i}$.
We take the generic point $\xi$ of an irreducible component of $W_{j}$
having the maximal dimension. We put $X_{j}:=U_{\xi}\cap W_{j}$ and
$W_{j+1}:=W_{j}\setminus X_{j}$. We then have 
\[
(\dim W_{i},\nu(W_{i}))>(\dim W_{i+1},\nu(W_{i+1})),
\]
where pairs are ordered lexicographically and $\nu(-)$ means the
number of irreducible components of maximal dimension. The procedure
ends after finitely many steps and gives a desired decomposition of
$X$.
\end{proof}
\begin{defn}
We say that a morphism $f\colon Y\to X$ of $k$-varieties is a\emph{
very weak $\AA^{m}$-bundle} if for every geometric point $x\in X(K)$,
the fiber $f^{-1}(x)$ is universally homeomorphic to a weak $\AA_{K}^{m}$. 
\end{defn}

In particular, a universal homeomorphism of $k$-varieties is a very
weak $\AA^{0}$-bundle. 
\begin{cor}
\label{cor:model}There exists a scheme morphism $f\colon Y\to X$
which is a very weak $\AA^{n(l+1)-2-m}$-bundle and induces the $P$-morphism
$\Eis_{l}^{(m)}/\GG_{m}\to\Delta_{n}^{(m)}$.
\end{cor}

\begin{proof}
From Proposition \ref{prop:closed}, the P-morphism $h\colon\Eis_{l}^{(m)}/\GG_{m}\to\Delta_{n}^{(m)}$
satisfies the assumption of Lemma \ref{lem:technical}. Therefore,
if $\Gamma$ denotes its graph, then there exists a decomposition
$\Delta_{n}^{(m)}=\bigsqcup_{i=1}^{l}X_{i}$ into locally closed subsets
$X_{i}$ such that for each $i$, $\pr_{\Delta_{n}^{(m)}}^{-1}(X_{i})\cap\Gamma$
is a locally closed subset. Giving these locally closed subsets the
reduced scheme structures, let us consider the coproducts $X:=\coprod_{i=1}^{l}X_{i}$
and $Y:=\coprod_{i=1}^{l}(\pr_{\Delta_{n}^{(m)}}^{-1}(X_{i})\cap\Gamma)$.
The natural morphism $f\colon Y\to X$ induces the P-morphism $h\colon\Eis_{l}^{(m)}/\GG_{m}\to\Delta_{n}^{(m)}$.
From construction, the morphism
\[
\pr_{\Delta_{n}^{(m)}}^{-1}(X_{i})\cap\Gamma\to h^{-1}(X_{i})
\]
is a universal homeomorphism. From Lemma \ref{lem:closed-irreducible},
this is a very weak $\AA^{n(l+1)-2-m}$-bundle.
\end{proof}

\section{Motivic mass formulas}

Let $K_{0}(\Var/k)$ denote the Grothendieck ring of $k$-varieties
and let $\LL:=[\AA_{k}^{1}]$. 
\begin{defn}
\label{def:Grothendieck ring}We define $K_{0}^{\heartsuit}(\Var/k)$
to be the quotient of $K_{0}(\Var/k)$ by the following relation:
for each very weak $\AA^{m}$-bundle $Y\to X$, we have $[Y]=[X]\LL^{m}$.
We define $\cM_{k}^{\heartsuit}:=K_{0}^{\heartsuit}(\Var/k)$ to be
the localization of $K_{0}^{\heartsuit}(\Var/k)$ by $\LL$. For $s\in\ZZ$,
let $F_{s}\subset\cM_{k}^{\heartsuit}$ to be the subgroup generated
by elements of the form $[X]\LL^{r}$ with $\dim X+r\le-s$. We define
the \emph{dimensional completion} $\widehat{\cM}_{k}^{\heartsuit}$
of $\cM_{k}^{\heartsuit}$ to be the projective limit $\varprojlim\cM_{k}^{\heartsuit}/F_{s}$. 
\end{defn}

The ring structure on $K_{0}(\Var/k)$ induces ones on $K_{0}^{\heartsuit}(\Var/k)$,
$\cM_{k}^{\heartsuit}$ and $\widehat{\cM}_{k}^{\heartsuit}$. 
\begin{rem}
Let us suppose now that $k$ is a finitely generated field. We denote
by $\mathbf{WRep}_{G_{k}}(\QQ_{l})$ the abelian category of mixed
$l$-adic representations of the absolute Galois group $G_{k}=\Gal(k^{\sep}/k)$.
Its Grothendieck ring $K_{0}(\mathbf{WRep}_{G_{k}}(\QQ_{l}))$ admits
the completion $\widehat{K}_{0}(\mathbf{WRep}_{G_{k}}(\QQ_{l}))$
with respect to weights. We have a ring homomorphism $\cM_{k}\to K_{0}(\mathbf{WRep}_{G_{k}}(\QQ_{l}))$
sending $[X]$ to $\sum_{i}(-1)^{i}[H_{c}^{i}(\overline{X},\QQ_{l})]$,
which extends to $\widehat{\cM}_{k}^{\heartsuit}\to\widehat{K}_{0}(\mathbf{WRep}(\QQ_{l}))$
by an argument similar to one in \cite[Lemma 9.11]{yasuda2024motivic}
and its subsequent paragraphs. Hence, for each equality in $\widehat{\cM}_{k}^{\heartsuit}$
that we obtain below, we have the corresponding equality in $\widehat{K}_{0}(\mathbf{WRep}(\QQ_{l}))$,
provided that $k$ is finitely generated.
\end{rem}

\begin{lem}
\label{lem:motivic-class}We have 
\[
[\Delta_{n}^{(m)}]\LL^{n(l+1)-2-m}=[\Eis_{l}^{(m)}/\GG_{m}]\in K_{0}^{\heartsuit}(\Var/k).
\]
\end{lem}

\begin{proof}
Take a morphism $Y\to X$ as in Corollary \ref{cor:model}. We have
the following equalities in $K_{0}^{\heartsuit}(\Var/k)$:
\[
[\Eis_{l}^{(m)}/\GG_{m}]=[Y]=[X]\LL^{n(l+1)-2-m}=[\Delta_{n}^{(m)}]\LL^{n(l+1)-2-m}.
\]
\end{proof}
The map $\Eis_{l+1}/\GG_{m}\to\Eis_{l}/\GG_{m}$ is an $\AA^{n}$-bundle.
We define a motivic measure $\mu$ on $\Eis/\GG_{m}$ taking values
in $\widehat{\cM}_{k}^{\heartsuit}$ in the usual way: for a cylinder
$C\subset\Eis/\GG_{m}$, $\mu(C):=[\pi_{l}(C)]\LL^{-nl}$ for $l\gg0$.
We then extend this measure to measurable subsets (see \cite[Chapter 6, Section 3]{chambert-loir2018motivic}
for details). In particular, $\Eis/\GG_{m}$, $\Eis^{\sep}/\GG_{m}$
and $\Eis^{(m)}/\GG_{m}$ are measurable subsets. The set $(\Eis/\GG_{m}$ 

and we have 
\begin{align*}
\mu(\Eis/\GG_{m}) & =\mu(\Eis^{\sep}/\GG_{m})=\sum_{m}\mu(\Eis^{(m)}/\GG_{m})\\
 & =[\pi_{1}(\Eis)/\GG_{m}]\LL^{-n}=\LL^{-1}.
\end{align*}

\begin{thm}
\label{thm:motivic-Serre}We have
\begin{equation}
\int_{\Delta_{n}^{\circ}}\LL^{-\bd}=\LL^{-n+1}\in\widehat{\cM}_{k}^{\heartsuit}.\label{eq:Serre}
\end{equation}
\end{thm}

\begin{proof}
In $\widehat{\cM}_{k}^{\heartsuit}$, 
\begin{align*}
[\Delta_{n}^{(m)}]\LL^{-m} & =[\Eis_{l}^{(m)}/\GG_{m}]\LL^{-n(l+1)+2}\\
 & =\mu(\Eis^{(m)}/\GG_{m})\LL^{-n+2}.
\end{align*}
It follows that 
\begin{align*}
\int_{\Delta_{n}^{\circ}}\LL^{-d} & =\LL^{-n+2}\sum_{m}\mu(\Eis^{(m)}/\GG_{m})\\
 & =\LL^{-n+2}\mu(\Eis/\GG_{m})\\
 & =\LL^{-n+1}\in\widehat{\cM}_{k}^{\heartsuit}.
\end{align*}
\end{proof}
\begin{cor}
\label{cor:motivic-Bhargava}Let $P(n,l)$ denote the number of partitions
of $n$ into exactly $l$ parts;
\[
P(n,l):=\{(q_{1},\dots,q_{l})\in(\ZZ_{>0})^{l}\mid q_{1}\ge\cdots\ge q_{l},\,\sum_{i=1}^{l}q_{i}=n\}.
\]
Then, we have
\begin{equation}
\int_{\Delta_{n}}\LL^{-d}=\sum_{j=0}^{n-1}P(n,n-j)\LL^{-j}\in\widehat{\cM}_{k}^{\heartsuit}.\label{eq:Bhargava}
\end{equation}
\end{cor}

\begin{proof}
The natural P-morphism 
\[
\coprod_{l=1}^{n}\coprod_{(q_{i})\in P(n,l)}\prod_{i=1}^{l}\Delta_{q_{i}}^{\circ}\to\Delta_{n},\,(A_{i})_{1\le i\le l}\mapsto\prod_{i=1}^{l}A_{i}
\]
is geometrically bijective. Moreover, we have
\[
\bd(\prod_{i}A_{i})=\sum_{i}\bd(A_{i}).
\]
Thus
\[
\int_{\prod_{i=1}^{l}\Delta_{q_{i}}^{\circ}}\LL^{-\bd}=\prod_{i=1}^{l}\int_{\Delta_{q_{i}}^{\circ}}\LL^{-\bd}=\prod_{i=1}^{l}\LL^{-q_{i}+1}=\LL^{-n+l}
\]
and 
\begin{align*}
\int_{\Delta_{n}}\LL^{-d} & =\sum_{l=1}^{n}\sum_{(q_{i})\in P(n,l)}\int_{\prod_{i=1}^{l}\Delta_{q_{i}}^{\circ}}\LL^{-\bd}\\
 & =\sum_{l=1}^{n}\sum_{(q_{i})\in P(n,l)}\LL^{-n+l}\\
 & =\sum_{l=1}^{n}P(n,l)\LL^{-n+l}\\
 & =\sum_{j=0}^{n-1}P(n,n-j)\LL^{-j}.
\end{align*}
\end{proof}
\begin{thm}
\label{thm:motivic-Krasner}In what follows, we follow the convention
that $0\nmid n$ for every positive integer $n$.
\begin{enumerate}
\item $\Delta_{n}^{(0)}$ is P-isomorphic to $\Spec k$.
\item We have $\Delta_{n}^{(m)}=\emptyset$ for $m>0$ satisfying either
of the following conditions:
\begin{enumerate}
\item $p\nmid n,\,m\ne n-1$,
\item $p\mid n,\,p\nmid(m-n+1),\,m-n+1<0$,
\item $p\mid n,\,p\mid(m-n+1)$.
\end{enumerate}
\item If $p\nmid n$ and $m=n-1$, then $\Delta_{n}^{(m)}$ is P-isomorphic
to $\Spec k$.
\item If $p\mid n$, $p\nmid(m-n+1)$ and $m-n+1\ge0$, then we have 
\[
[\Delta_{n}^{(m)}]=(\LL-1)\LL^{\lfloor(m-n+1)/p\rfloor}
\]
in $\cM_{k}^{\heartsuit}[(\LL-1)^{-1}]$ as well as in $\widehat{\cM}_{k}^{\heartsuit}$. 
\end{enumerate}
\end{thm}

\begin{proof}
(1) For an algebraically closed field $K$, $\Delta_{n}^{(0)}(K)$
is a singleton corresponding to the trivial étale algebra $K\tpars^{n}/K\tpars$,
which shows the assertion.

(2) If $A_{y}/K\tpars$ corresponds to an Eisenstein polynomial $x^{n}+y_{1}x^{n-1}+\cdots+y_{n}$,
then from (\ref{eq:disc}), we have 
\[
\bd_{A_{y}}=n\min_{0\le i\le n}\ord\left(\sum_{i=0}^{n}(n-i)y_{i}\varpi^{n-i-1}\right).
\]
If $p\nmid n$, then the minimum on the right side is attained by
$\ord n\varpi^{n-1}=(n-1)/n$ and $\bd_{A_{y}}=n-1$, which shows
(a). (The case (a) follows also from the proof of (3).) Suppose $p\mid n$
and $A/K\tpars$ is a totally ramified extension of degree $n$ and
discriminant exponent $m$. From Proposition \ref{prop:Eis-m-explicit},
for some $i$ with $p\nmid i$ and $0\le i\le n-1$, we have
\[
m=n-i-1+n\ord y_{i}.
\]
It follows that 
\[
m-n+1=i-n\ord y_{i}>0
\]
and 
\[
p\nmid(m-n+1).
\]
This shows (b) and (c). 

(3) Let $K$ be an algebraically closed field and let $A/K\tpars$
be an extension of degree $n$. Let $\widetilde{A}$ be its Galois
closure with Galois group $G$. We may identify with a transitive
subgroup of $S_{n}$. Moreover, $G$ is isomorphic to the semidirect
product $H\rtimes C$ of a $p$-group $H$ and a tame cyclic subgroup
$C$ if $p>0$ and isomorphic to a cyclic group if $p=0$. We claim
that if $p>0$, then $H=1$. To show this by contradiction, suppose
that $H\ne1$. Since $p\nmid n$ and $H$ is a $p$-group, the $H$-action
on $\{1,\dots,n\}$ has at least one fixed point, say 1. Since $G$
is transitive, there exists $g\in G$ such that $g(1)$ is not fixed
by $H$. Then, $g^{-1}Hg\ne H$, which contradicts the fact that $H$
is a normal subgroup of $G$. We have proved the claim. Then, $G=C$
is the cyclic subgroup of $S_{n}$ generated by a cyclic permutation
of an $n$-cycle. In particular, the stabilizer $\Stab(1)$ of $1\in\{1,\dots,n\}$
is trivial and $A=\widetilde{A}^{G}=\widetilde{A}$. We conclude that
$A/K\tpars$ is a cyclic Galois extension. As is well-known, $A$
is isomorphic to $K\llparenthesis t^{1/n}\rrparenthesis$. Hence $\Delta_{n}^{(m-1)}(K)$
is a singleton, which shows the assertion.

(4) We put $c:=m-n+1$. From Lemma \ref{lem:motivic-class}, we have
\begin{align*}
[\Delta_{n}^{(m)}] & =[\Eis_{l}^{(m)}/\GG_{m}]\LL^{-nl+c+1}\in K_{0}^{\heartsuit}(\Var_{k}).
\end{align*}
Since the $\GG_{m}$-torsor $\Eis_{l}^{(m)}\to\Eis_{l}^{(m)}/\GG_{m}$
is Zariski locally trivial thanks to Hilbert's Theorem 90, we have
\[
[\Eis_{l}^{(m)}/\GG_{m}]=(\LL-1)^{-1}[\Eis_{l}^{(m)}]
\]
in $\cM_{k}^{\heartsuit}[(\LL-1)^{-1}]$. From Corollary \ref{cor:pi_l(Eis^m)},
\[
[\Eis_{l}^{(m)}]=(\LL-1)^{2}\LL^{nl-c+\lfloor c/p\rfloor-1}.
\]
Combining these equalities, we get
\begin{align*}
[\Delta_{n}^{(m)}] & =(\LL-1)^{-1}\left((\LL-1)^{2}\LL^{(nl-c+\lfloor c/p\rfloor-1)}\right)\LL^{-nl+c+1}\\
 & =(\LL-1)\LL^{\lfloor c/p\rfloor}
\end{align*}
in $\cM_{k}^{\heartsuit}[(\LL-1)^{-1}]$. Since $\LL-1$ is invertible
in $\widehat{\cM}_{k}^{\heartsuit}$, we have a natural homomorphism
$\cM_{k}^{\heartsuit}[(\LL-1)^{-1}]\to\widehat{\cM}_{k}^{\heartsuit}$.
Thus, the same equality holds also in $\widehat{\cM}_{k}^{\heartsuit}$.
\end{proof}
The following corollary is a direct consequence of the last theorem:
\begin{cor}
We have
\[
\dim\Delta_{n}^{(m)}=\begin{cases}
0 & (p\nmid n,\,m=n-1),\\
-\infty & (p\nmid n,\,m\ne n-1),\\
\lceil(m-n+1)/p\rceil & (p\mid n,\,p\nmid(m-n+1),\,m-n+1\ge0),\\
-\infty & (p\mid n,\,p\nmid(m-n+1),\,m-n+1<0),\\
-\infty & (p\mid n,\,p\mid(m-n+1)).
\end{cases}
\]
Here we follow the convention that $\dim\emptyset=-\infty$ and that
if $p=0$, then $p\nmid n$. Moreover, when $\Delta_{n}^{(m)}\ne\emptyset$,
then it has only one irreducible component of the maximal dimension. 
\end{cor}

\begin{rem}
The original mass formulas in \cite{krasner1962nombredes,krasner1966nombredes,serre1978unetextquotedblleftformule,bhargava2007massformulae}
also hold for local fields of characteristic zero. It would be possible
to similarly prove motivic mass formulas in characteristic zero, once
the relevant P-moduli space will be constructed.
\end{rem}

\bibliographystyle{alpha}
\bibliography{MotivicMass}

\end{document}